\documentclass[a4paper,10pt]{article}
\usepackage[utf8]{inputenc}
\usepackage{amsmath}
\usepackage{amsthm}
\usepackage{amssymb}
\usepackage{hyperref}\hypersetup{
    colorlinks=true,
    citecolor=blue,
    filecolor=magenta,      
    urlcolor=cyan,
}
\usepackage{pdfsync}
\usepackage{textcomp}
\usepackage{xcolor}
\usepackage{ulem}
\usepackage{authblk}

\newtheorem{lem}{Lemma}[section]

\newtheorem{thm}{Theorem}[section]
\newtheorem{prop}[thm]{Proposition}
\newtheorem{cor}{Corollary}
\numberwithin{equation}{section}

\theoremstyle{definition}
\newtheorem{defn}{Definition}[section]

\theoremstyle{theorem}
\newtheorem{rmk}{Remark}[section]

\title{PI controllers for the general Saint-Venant equations}
\author[*]{Amaury Hayat}
\affil[*]{Centre d’Enseignement et de Recherche en Mathématiques et Calcul Scientifique, \'{E}cole des Ponts ParisTech, 6 - 8, Avenue Blaise Pascal, Cité Descartes—Champs sur Marne, 77455 Marne la Vall\'{e}e, France. E-mail: \texttt{amaury.hayat@enpc.fr.}}
\date{\empty}
\begin{document}

\maketitle

\begin{abstract}

We study the exponential stability in the $H^{2}$ norm of the nonlinear Saint-Venant (or shallow water) 
equations with arbitrary friction and slope 
using a single Proportional-Integral (PI) control at one end of the channel. 
\textcolor{black}{Using a good but simple Lyapunov function} we find a simple and explicit condition on the gain of the PI control 
to ensure the exponential stability of any steady-states. 
This condition is independent of the slope, the friction coefficient, the length of the river, the inflow disturbance and, more surprisingly, \textcolor{black}{can be made independent of} the steady-state considered.
When the inflow disturbance is time-dependent and no steady-state exist, we still have the Input-to-State stability of the system,
and we show that changing slightly the PI control enables to recover the exponential stability of slowly varying trajectories.
\end{abstract}

\tableofcontents
\newpage
\section*{Introduction}
Discovered in 1871, the Saint-Venant equations \cite{SaintVenant} (or 1-D shallow water equations) are among the most famous equations in fluid dynamics and have been investigated in hundreds of studies. 
Although being quite simple, their richness has made them become a major tool in practice for many industrial goals, the most famous being probably the regulation of navigable rivers. 
They are the ground model for such purpose in France and Belgium.
Regulation of rivers 
is a major issue, for navigation, freight transport, 
renewable energy production,
but also for safety reasons, especially as
several nuclear plants all around the world are implanted close to rivers.
For these reasons, the stability of the steady-states of the Saint-Venant equations has been, and is still, a major issue.\\

Many results were obtained in the last decades. 
In 1999, the robust stability of the homogeneous linearized Saint-Venant equations was shown using a Lyapunov approach and proportional feedback controllers \cite{coron1999}.
Later, the stability of the homogeneous nonlinear Saint-Venant equations was achieved, still using proportional feedback controllers.
In 2008, through a semi-group approach \cite{dos2008}, the stability of the inhomogeneous nonlinear Saint-Venant equation was shown for sufficiently small friction and slope (or equivalently sufficiently small canal), and these results were successfully applied to real data sets from the Sambre river in Belgium.
More recently, in \cite{BC2017} the authors have given sufficient conditions to stabilize the nonlinear Saint-Venant equations with arbitrary friction for the $H^{2}$ norm but no slope using again proportional feedback controllers, 
and in \cite{HS} with both arbitrary friction and slope.
This last result is proved by exhibiting \textcolor{black}{a Lyapunov function that has a simple form close to a local entropy} for the nonlinear inhomogeneous Saint-Venant equations. 

It is worth mentioning that other stability results have also been obtained in less classical cases or with less classical feedbacks. 
For instance, in \cite{stvenantjump} was shown the rapid stabilization of the homogeneous nonlinear Saint-Venant equations when a shock (e.g. a hydraulic jump) occurs in the target steady-state.
Such \textcolor{black}{a} shock induces new difficulties and the presence of shocks can limit in general the controllability and the stability in weaker norms of hyperbolic systems 
with boundary controls
\cite{ABG,BG}.
Also, several results (e.g. 
\cite{CVBK})
were obtained using a backstepping approach, a very powerful method based on a Volterra transformation, developed mainly for PDE in \cite{krsticbook},
\textcolor{black}{and generalized recently with a Fredholm transformation for hyperbolic systems \cite{CoronFredholm,Zhang,Zhang2}.}
One may look at \cite{HS} for a more detailed survey about this method and its use for the Saint-Venant equations.
However, backstepping gives rise to non-local and non-static feedback laws that are likely to be harder to implement, and, to our knowledge, have not been implemented yet.\\

Most of the previous results were performed with static proportional feedback controllers. When it comes to industrial applications, however, the proportional integral (PI) control is by far the most popular regulator. 
It is used for instance for the regulation of the Sambre and Meuse river in Belgium \cite[Chapter 8]{BastinCoron1D}.
The reason behind such preference is the robustness of the PI control with off-set errors \cite[Chap. 11.3]{Astrom2}. \textcolor{black}{An example can be found in \cite{DosSantosBastinCoron} 
where the authors show the interest of adding an integral term to a proportional control on a linear and homogeneous system, and exhibit coherent experimental result.}

For these reasons, the PI controller has fed a wide literature, at least when used on finite dimensional systems.
However, despite their indisputable practical interest, PI controllers for nonlinear infinite dimensional systems have shown hard to handle mathematically 
and even studying simple systems give sometimes rise to lengthy proofs with relatively sophisticated tools \cite{PItransport}.
While the behaviour and the stability of linearized equations
with PI controller
has been well understood in the past, partly thanks to spectral tools like the spectral mapping theorem (e.g. \cite{Neves,Lichtner} for hyperbolic systems), no such tools exist for nonlinear systems and
the stability of the nonlinear Saint-Venant equations has remained a challenge until today.
\textcolor{black}{Among the existing linear result using a spectral approach, one can refer to \cite{XuSallet,XuSallet2014} where the authors find a sufficient condition for the stabilization of the linearized inhomogeneous Saint-Venant equations.
Necessary and sufficient conditions for the linearized homogeneous Saint-Venant equations
are given in \cite[Section 2.2.4.1, 3.4.4]{BastinCoron1D}.
In \cite{CoronTamasoiu2015} the authors find a necessary and sufficient condition for a linear scalar equation and show the difficulty of finding good conditions for the nonlinear equation, while 
in \cite{BastinCoronTamasoiu2015} the authors deal with $2\times2$ systems.}
Among the existing nonlinear results 
\textcolor{black}{one can refer to \textcolor{black}{\cite{TerrandAndrieuDos} in the case where the operator without PI control generates an exponentially stable semi-group}, 
\cite{Trinh2015} where the authors find a sufficient condition for the nonlinear homogeneous Saint-Venant equations,
\cite[2.2.4.2]{BastinCoron1D} where the authors find a necessary and sufficient condition also for the 
nonlinear homogeneous Saint-Venant equations,
while 
\cite[Section 5.4.4,5.5]{BastinCoron1D} and \cite{BastinCoron2013} give a sufficient condition for the inhomogeneous Saint-Venant equations for a single channel or a network, 
but in the particular case of constant steady-states only, which simplifies their analysis \cite{C1}. Strictly speaking, this last result was derived for the linearized system 
but with a Lyapunov approach, which can easily be generalized to the nonlinear system.}
More recently, and this is the most advanced result yet, \cite{BastinCoronPI} gave a sufficient condition of stability for the inhomogeneous Saint-Venant equations with an arbitrary friction and river length but only
in the absence of slope, using a Lyapunov approach.\\

In this paper, we consider the stabilization of the general nonlinear Saint-Venant equations with a single boundary PI control.
We give a simple and explicit condition on the parameters of the PI controller such that any steady-state
is exponentially stable for the $H^{2}$ norm. 
While stability results in inhomogeneous and nonlinear systems often imply a limit length for the domain, depending on the source term, above with we are unable to guarantee any stability (\cite{C1,C1_22,BastinCoron22,dos2008} or \cite[Chap. 6]{BastinCoron1D}), 
this result holds whatever the friction, the slope, and the length of the channel. 
Besides, our condition is independent of the slope, the friction coefficient, the river length, and, more surprisingly, \textcolor{black}{can be made independent of }the steady-state considered. 
Finally, when there is no slope this condition is less restrictive that the condition obtained in \cite{BastinCoronPI} 
and when there is no friction or slope this condition coincides with the \textcolor{black}{necessary and sufficient} spectral condition of stability for the linearized system given in \cite{BastinCoronTamasoiu2015} and \cite[Theorem 2.7]{BastinCoron1D}.

The case where the inflow disturbances are time dependent and no steady-states exists was seldom considered in the literature.
However, it is in fact unlikely that the industrial target state is a real steady-state as the inflow disturbance often depends on time in practice,
even though only slowly. 
Therefore, 
in the more general framework of slowly time-varying target states, 
we show the Input-to-State Stability (ISS) of the system with respect to the variation of the inflow disturbance. 
Finally, we show that if 
we allow the controller to depend on the target state, 
by changing slightly the PI controller, 
we can ensure
the exponential stability of slowly-varying target trajectories that are the natural target trajectories to consider when there is no steady-state of the system.

This paper is organized as follows: in Section \ref{s1} we give a description of the nonlinear Saint-Venant equations, we introduce the time-varying target trajectories together with some definitions and existence results, then we state our main results. 
In Section \ref{s2} we prove our main result, Theorem \ref{th1}, that deals with the exponential stability of time-varying state. 
In the Appendix, we show that Corollary \ref{cor1} dealing with the exponential stability of steady-states, 
and Theorem \ref{th2} showing the ISS of the system with respect to the variation of the inflow disturbance,
are both deduced from the proof of Theorem \ref{th1}. 

\section{Model description}
\label{s1}
We 
consider
the following nonlinear Saint-Venant equations for a rectangular channel with arbitrary slope and friction.
\begin{equation}
 \begin{split}
&\partial_{t}H+\partial_{x}(HV)=0,\\
&\partial_{t}V+V\partial_{x}V+g\partial_{x}H+\left(\frac{kV^{2}}{H}-C(x)\right)=0.
\end{split}
\label{Stvenant}
\end{equation}
Here, $k$ is an arbitrary nonnegative friction coefficient and $C$ denotes the slope, which is assumed to be a $C^{2}$ function, with $C(x):=-gdB/dx$ where $B$ is the bathymetry and $g$ the acceleration of gravity.
We are interested in systems where the water flow uphill is a given function, unknown and imposed by external conditions, for instance a flow coming from another country, while the water flow downhill is controlled through a hydraulic installation. Therefore, we have the following boundary conditions,
\begin{equation}
\begin{split}
H(t,0)V(t,0)&=Q_{0}(t),\\
H(t,L)V(t,L)&=U(t),
\end{split}
\label{boundary02}
\end{equation} 
where $U(t)$
is a control feedback and $Q_{0}(t)$
is the incoming flow, which is a given (and unknown) function. Here $L$ denotes the length of the water channel.
In practical situations, the formal control $U(t)$ can be expressed by a simple linear model \cite{BastinCoronPI}
\begin{equation}
U(t)=v_{G}(H(t,L)-U_{1}(t))\textcolor{black}{,}
\label{vG}
\end{equation} 
where $U_{1}(t)$ is the elevation of the gate of the dam, which is the real control input that can be chosen, while $v_{G}$ is a constant depending on the parameters of the gate (potentially unknown as well).\\

\subsection{Control goal and target trajectory}
Usually, the industrial goal of such system is to stabilize the level of the water at the end point $H(t,L)$, called control point, to a target value $H_{c}>0$. 
On the other hand, the usual mathematical goal in such a problem is to stabilize a target steady-state $(H^{*},V^{*})$, potentially nonuniform \cite{BastinCoron1D}[Preface]. However, in the present problem \eqref{Stvenant}--\eqref{boundary02}, it is clear that, when $Q_{0}$ 
is not constant, it is impossible to aim at stabilizing any steady-state and one needs to aim at stabilizing other target trajectories.
Therefore, we define the following target trajectory $(H_{1},V_{1})$ that we aim to stabilize as the solution of
\begin{equation}
\begin{split}
&\partial_{t}H_{1}+\partial_{x}(H_{1}V_{1})=0,\\
&\partial_{t}V_{1}+V_{1}\partial_{x}V_{1}+g\partial_{x}H_{1}+\left(\frac{kV_{1}^{2}}{H_{1}}-C(x)\right)=0,\\
&H_{1}(t,0)V_{1}(t,0)=Q_{0}(t),\\
&H_{1}(t,L)=H_{c},\\
\end{split}
\label{target}
\end{equation}
with the initial condition
\begin{equation}
H_{1}(0,\cdot)=H^{*}(\cdot)\text{ and }V_{1}(0,\cdot)=V^{*}(\cdot),
\label{initialtarget}
\end{equation} 
where $(H^{*},V^{*})$ is the (unique) steady-state solution of the system \textcolor{black}{when $Q_{0}$ is} constant, equal to $Q_{0}(0)$. Namely, $(H^{*},V^{*})$ is the solution of
\begin{equation}
\begin{split}
 &\partial_{x}(HV)=0,\\
 &V\partial_{x}V+g\partial_{x}H+\left(\frac{kV^{2}}{H}-C(x)\right)=0,\\
 &H(L)=H_{c},
 \end{split}
 \label{targetsteady0}
 \end{equation}
with condition at $x=0$
\begin{equation}
 \begin{split}
 &H^{*}(0)V^{*}(0)=Q_{0}(0).
 \end{split}
\label{targetsteady}
\end{equation}
We are now going to show that the trajectory $(H_{1},V_{1})$ exists for any time and satisfies some bounds.

\paragraph{Existence and bounds of the target trajectory $(H_{1},V_{1})$}

Instead of studying directly our target trajectory $(H_{1},V_{1})$ we first construct an intermediary family of functions 
$(H_{0},V_{0})$ where at each time $t$ $(H_{0}(t,\cdot),V_{0}(t,\cdot))$ is defined as the space dependent steady-state that would be associated with the constant flow $Q_{0}(t)$. This is detailed in the following paragraph.\\

We defined previously $(H^{*}, V^{*})$ as the steady-state associated to a constant flux $Q_{0}\equiv Q_{0}(0)$. This means that. $(H^{*},V^{*})$ is the solution of the ODE problem \eqref{targetsteady0} with 
initial condition $H^{*}(0)V^{*}(0)=Q_{0}(0)$. But in fact
at each time $t^{*}\in\mathbb{R}^{*}_{+}$, we can also define a steady-state $(H^{*}_{t^{*}},V^{*}_{t^{*}})$ associated to a constant flux $Q_{0}\equiv Q_{0}(t^{*})$. This means that
$(H^{*}_{t^{*}},V^{*}_{t^{*}})$ is the solution of the ODE problem \eqref{targetsteady0} 
with initial condition satisfying
\begin{equation}
\begin{split}
 &H_{t^{*}}^{*}(0)V_{t^{*}}^{*}(0)=Q_{0}(t^{*}).
 \end{split}
 \label{targetsteady2}
\end{equation} 
\textcolor{black}{Although the system \eqref{targetsteady0}--\eqref{targetsteady2} could seem peculiar as it has boundary conditions imposed both in $0$ and in $L$, 
we know looking at the first equation of \eqref{targetsteady0} that this system \eqref{targetsteady0}, \eqref{targetsteady2} is in fact equivalent to a single ODE on $H_{t^{*}}^{*}$
with boundary condition $H_{t^{*}}^{*}(L)=H_{c}$ and $V_{t^{*}}^{*}$ defined by $V_{t^{*}}^{*}=Q_{0}(t)/H_{t^{*}}^{*}$.
Thus} for each $t^{*}\in[0,+\infty)$ such function exists on $[0,L]$, is unique and $C^{3}$ provided that the state stays in the fluvial regime (or subcritical regime), i.e. $g H_{t^{*}}^{*}>V^{*2}_{t^{*}}$ on $[0,L]$,
\textcolor{black}{which, for a given $H_{c}$, is equivalent to a bound on $Q_{0}(t^{*})$}
(see \cite{HS} for more details).
As we are interested in \textcolor{black}{stabilizing} physical trajectories in the fluvial regime, we assume that this assumption is satisfied \textcolor{black}{in the following} and that
there exist $\alpha>0$ and $H_{\max}>0$ independent of $t^{*}\in[0,\infty)$ such that
\begin{equation}
\begin{split}
H_{t^{*}}^{*}<\frac{1}{2}H_{\max}\text{  on  }[0,L],\\
gH_{t^{*}}^{*}-V_{t}^{*2}>2\alpha\text{  on  }[0,L].
\end{split}
\label{fluvial0}
\end{equation} 
\textcolor{black}{
For a given $H_{c}$, this is again equivalent to imposing a bound $Q_{\infty}$ on $\lVert Q_{0}\rVert_{L^{\infty}(0,\infty)}$, from \eqref{targetsteady0} and \eqref{targetsteady2},
which would be more logical.
However, for convenience, we will still use $H_{\max}$ and $\alpha$ in the following. This assumption is quite physical, especially as in practical situation the river is in fluvial regime and $Q_{0}(t)$ is often periodic or quasi-periodic.}
This gives a family of one-variable functions indexed by a parameter $t^{*}$, which can also be seen as the two-variable functions 
\begin{equation}
(H_{0},V_{0}) : (t,x)\rightarrow (H^{*}_{t}(x),V^{*}_{t}(x)).
\end{equation}
Besides, from \eqref{targetsteady}, as $(H_{t}^{*},V_{t}^{*})$ 
is the solution of a system of ODE with a parameter $t$, the two variable functions $(H_{0},V_{0})$ therefore belongs to $C^{3}([0,+\infty)\times(0,L))$ (see \cite{Hartman}[Chap. 5, Cor. 4.1]). 
And from its definition, one can note that $(H_{0}(0,\cdot),V_{0}(0,\cdot))=(H^{*},V^{*})$.\\

For clarity, we summarize here the different families of functions we introduced.
\begin{itemize}
\item 
$(H^{*},V^{*})$, a function of $x$, the steady-state of the system when $Q_{0}\equiv cste$\\
\item
$(H_{1},V_{1})$, a function of $t$ and $x$, the target trajectory to reach when $Q_{0}$ is not a 
constant. This trajectory is compatible with the objective $H(t,L) = H_{c}$, for any $t\in [0,T]$.\\
\item 
$(H^{*}_{t^{*}},V^{*}_{t^{*}})$, a function of $x$, the steady-state of the system when $Q_{0}$ is a constant equal to $Q_{0}(t)$ ($t^{*}$ is fixed).\\
\item
$(H_{0},V_{0})$, a function of $t$ and $x$, the family such that $(H_{0}, V_{0})\;:\;(t,x) t\rightarrow (H_{t}^{*}(x),V^{*}_{t}(x))$
\end{itemize}

Now that we have introduced this intermediary family of functions, we can show the existence of the target trajectory $(H_{1},V_{1})$ and 
we have the following Input-to-State Stability (ISS) result \textcolor{black}{(see \cite{Sontag} for a definition of ISS for finite dimensional systems, 
\cite[Chap 1, Chap 3]{IassonKrstic} for a generalization to first-order hyperbolic PDE and \cite{PrieurISS} for the use of Lyapunov function to achieve ISS on time-varying hyperbolic systems)},
\begin{prop}
Assume that $\partial_{t}Q_{0}\in C^{2}([0,\infty))$. There exist positive constants $c_{1}$, $c_{2}$,
$\mu>0$, $\nu>0$ and $\delta>0$ such that if
$\lVert\partial_{t}Q_{0}\rVert_{C^{2}([0,+\infty))}\leq\delta$, then for any $(H_{1}^{0},V_{1}^{0})\in H^{2}((0,L),\textcolor{black}{\mathbb{R}^{2}})$ such that
\begin{equation*}
\lVert H_{1}^{0}-H^{*}\rVert_{H^{2}(0,L)}+\lVert V_{1}^{0}-V^{*}\rVert_{H^{2}(0,L)}\leq \nu,
\end{equation*} 
the system \eqref{target} with initial condition $(H_{1}^{0},V_{1}^{0})$ has a unique solution $(H_{1},V_{1})\in C^{0}([0,+\infty),H^{2}(0,L))$ which satisfies the following ISS inequality
\textcolor{black}{\begin{equation}
\begin{split}
\lVert H_{1}(t,\cdot)-H_{0}(t,\cdot)\rVert_{H^{2}(0,L)}+&\lVert V_{1}(t,\cdot)-V_{0}(t,\cdot)\rVert_{H^{2}(0,L)}\\
\leq &c_{1}(\lVert H_{1}^{0}-H^{*}\rVert_{H^{2}(0,L)}+\lVert V_{1}^{0}-V^{*}\rVert_{H^{2}(0,L)})e^{-\frac{\mu t}{2}}\\
&+c_{2}\left(\int_{0}^{t}\left(|\partial_{t}Q_{0}(s)|+|\partial_{tt}^{2}Q_{0}(s)|+|\partial_{ttt}^{3}Q_{0}(s)|\right)e^{\frac{\mu s}{2}}ds\right)e^{-\frac{\mu t}{2}}.
\end{split}
\label{ISS}
\end{equation}}
\label{propISS}
\end{prop}
This result is shown in Appendix \ref{AppendixISS}, \textcolor{black}{and a definition of the $C^{2}$ norm is recalled in Remark \ref{rmkHpISS}}. Note that $Q_{0}$ is supposed to be bounded, which is quite physical, but there is no additional requirement on this bound besides the physical assumption given by $Q_{\infty}$ of remaining in the fluvial regime. This is important as in practical 
situations the value of the incoming flow can change a lot, even though slowly.\\

Here, we \textcolor{black}{choose} to stabilize the trajectory $(H_{1},V_{1})$ associated to $H_{1}^{0}=H^{*}$ and $H_{1}^{0}=V^{*}$. As we will see, this target trajectory can be seen as the natural trajectory to stabilize as it satisfies the industrial goal $H(t,L)=H_{c}$ 
and it coincides with the steady-state solution when $Q_{0}$ is a constant. 
\textcolor{black}{In this last case $Q_{0}$ and $H_{c}$ are imposed and $H^{*}$ and $V^{*}=Q_{0}/H^{*}$ are thus fully determined using \eqref{targetsteady0}.}
But one can note from \eqref{ISS} that, in fact, the behavior of $(H_{1},V_{1})$ at large time does not depend on the initial condition $(H_{1}^{0},V_{1}^{0})$ in \eqref{initialtarget}, provided that it is close in $H^{2}$ norm to $(H^{*},V^{*})$.\\
\begin{rmk}
The same ISS result can be shown replacing the $H^{2}$ norm in Proposition \ref{propISS} by the $H^{p}$ norm where $\textcolor{black}{p\in\mathbb{N}^{*}\setminus\{1\}}$, with the condition $\lVert\partial_{t}Q_{0}\rVert_{C^{p}([0,+\infty))}\leq\delta$ instead of $\lVert\partial_{t}Q_{0}\rVert_{C^{2}([0,+\infty))}\leq\delta$. This is shown in Appendix \ref{AppendixISS}. \textcolor{black}{We define here the $C^{p}$ norm for a function $U\in C^{p}(I)$, where $I$ is an interval, as
\begin{equation}
\rVert U \lVert_{C^{p}(I)}:=\max_{i\in[0,p]}(\lVert \partial_{t}^{i}U\rVert_{L^{\infty}(I)})
\end{equation}}

\label{rmkHpISS}
\end{rmk}

Thus, from Proposition \ref{propISS} and \eqref{fluvial0}, there exists a constant $\textcolor{black}{\delta}>0$ such that, if $\lVert\partial_{t}Q_{0}\rVert_{\textcolor{black}{C^{2}([0,\infty))}}<\textcolor{black}{\delta}$, then $(H_{1},V_{1})\in C^{0}([0,+\infty),H^{2}(0,L))$ and
\begin{gather}
H_{1}(t,x)<H_{\max},\text{   }\forall\text{  }(t,x)\in[0,+\infty)\times[0,L],
\label{Hinfty}\\
 gH_{1}(t,x)-V_{1}^{2}(t,x)>\alpha,\text{   }\forall\text{  }(t,x)\in[0,+\infty)\times[0,L].
 \label{fluvial}
\end{gather}

Besides, when $Q_{0}$ is a constant, it is easy to check that $(H_{0},V_{0})=(H^{*},V^{*})$ is also solution of \eqref{target}--\eqref{initialtarget}. Thus, from the uniqueness of the solution of \eqref{target}--\eqref{initialtarget}, 
$(H_{1},V_{1})=(H^{*},V^{*})$ and, therefore, we recover a steady-state.
This illustrates that $(H_{1},V_{1})$ can be seen as the natural target state
when $Q_{0}$ is not a constant anymore.
Moreover, from \eqref{target}, stabilizing $(H_{1},V_{1})$ also satisfies the industrial goal by stabilizing $H(t,L)$ on the value $H_{c}$.

\subsection{Control design and main result}
As mentioned in the introduction, a usual type of controller used in pratice to reach this aim is the proportional-integral (PI) controller. 
It has the advantage of eliminating the offset coming from constant load disturbances, which can usually appear in these systems
as the command on the gate's level are only known up to some constant incertainties.  
A generic PI controller is given by
\begin{equation}
U_{1}(t)=k_{p}(H_{c}-H(t,L))+k_{I}Z\textcolor{black}{,}
\label{defu1}
\end{equation}
where $k_{p}$ and $k_{I}$ are coefficients that can be designed and $Z$ accounts for the integral term, i.e.
\begin{equation}
\dot Z=H_{c}-H(t,L).
\label{Z}
\end{equation} 
With such controller, and using \eqref{vG}, the boundary conditions \eqref{boundary02} become \textcolor{black}{\eqref{Z} and}
\begin{equation}
\begin{split}
&H(t,0)V(t,0)=Q_{0}(t),\\
&H(t,L)V(t,L)=v_{G}(1+k_{p})H(t,L)-v_{G}k_{p}H_{c}-v_{G}k_{I}Z\textcolor{black}{,}
\label{boundary2}
\end{split}
\end{equation} 
 In Corollary \ref{cor1} we show that this boundary control can be used to stabilize \textcolor{black}{exponentially} a steady-state when $Q_{0}$ is a constant. 
In Theorem \ref{th2} we show that this control can also provide an Input-to-State Stability property with respect to $\partial_{t}Q_{0}$.
However, this control \eqref{defu1} cannot be used to stabilize \textcolor{black}{a} dynamic target trajectory $(H_{1},V_{1})$, as there is no function $Z_{1}\in C^{1}([0,+\infty))$ such that $(H_{1},V_{1},Z_{1})$ 
is a solution of \eqref{Stvenant}, \eqref{Z}, \eqref{boundary2} while $(H_{1},V_{1})$ is a solution of \eqref{target}. Therefore, when stabilizing a dynamic target trajectory, one has to add an additional term and use
\begin{equation}
U_{1}(t)=k_{p}(H_{c}-H(t,L))+k_{I}Z-f(t)\textcolor{black}{,} 
\label{defu12}
\end{equation} 
where $f(t):=H_{1}(t,L)V_{1}(t,L)/v_{G}$. The boundary conditions \eqref{boundary02} become then
\begin{equation}
\begin{split}
&H(t,0)V(t,0)=Q_{0}(t),\\
&H(t,L)V(t,L)=H_{1}V_{1}(t,L)+v_{G}(1+k_{p})(H(t,L)-H_{c})-v_{G}k_{I}Z,
\label{boundary1}
\end{split} 
\end{equation} 
where we have actually changed $Z$ and re-define $Z:=Z-k_{p}/k_{I}$, which still satisfies the equation \eqref{Z}. 

This new control \eqref{defu12} assumes that $V_{1}(t,L)$ is known at least up to a constant, as $H_{1}(t,L)=H_{c}$ and additional constants can be incorporated into $Z$. 
When no knowledge on the target state is available besides $H_{c}$, 
it is impossible to stabilize exponentially the system, and the best one can get is the Input-to-State Stability which is given by Theorem \ref{th2}. 
However in the following we will keep working with \eqref{defu12} and \eqref{boundary1} to show Theorem \ref{th1} and the exponential stability of the system, as the proof of Theorem \ref{th2} and Corollary \ref{cor1} which uses only the control \eqref{defu1} and \eqref{boundary2} are easily deduced from the proof of Theorem \ref{th1}.\\

We introduce the first-order compatibility conditions associated to the boundary conditions \eqref{boundary1} for an initial condition $(H^{0},V^{0},Z^{0})$.
\begin{equation}
\begin{split}
&H^{0}(0)V^{0}(0)=Q_{0}(0),\\
&H^{0}(L)V^{0}(L)=H_{1}V_{1}(0,L)+v_{G}(1+k_{p})(H^{0}(L)-H_{c})-k_{I}Z^{0},\\
&-\partial_{x}(H^{0}(0)V^{0}(0)+g\frac{gH^{0}(0)^{2}}{2})-(k(V^{0})^{2}(0)-CH^{0}(0))=Q'_{0}(0),\\
&-\partial_{x}(H^{0}(L)V^{0}(L)+g\frac{gH^{0}(L)^{2}}{2})-(k(V^{0})^{2}(L)-CH^{0}(L))=\partial_{t}(H_{1}V_{1})(0,L)\\
&\text{  }\text{  }\text{  }\text{  }\text{  }\text{  }\text{  }\text{  }\text{  }-v_{G}(1+k_{p})\partial_{x}(H^{0}(L)V^{0}(L))+k_{I}(H^{0}(L)-H_{c}).
\label{compat}
\end{split}
\end{equation} 
With such compatibility conditions the system \eqref{Stvenant}, \eqref{Z}, \eqref{boundary1} is well-posed and we have the following theorem due to Wang \cite{Wang}[Theorem 2.1]:
\begin{thm}[Well-posedness]
Let $T>0$,  and assume that $\lVert\partial_{t}Q_{0}\rVert_{C^{3}([0,+\infty))}\leq \delta(T)$, such that $(H_{1},V_{1})$ is well-defined and belongs to $C^{0}([0,T],H^{3}(0,L))$.  
There exists $\nu(T)>0$ such that for any $(H^{0},V^{0},Z^{0})\in (H^{2}((0,L))))^{2}\times \mathbb{R}$ satisfying 
\begin{equation}
\lVert H^{0}(\cdot)-H_{1}(0,\cdot)\rVert_{H^{2}(0,L)}+\lVert V^{0}(\cdot)-V_{1}(0,\cdot)\rVert_{H^{2}(0,L)}+|Z^{0}|\leq \nu(T),
\end{equation}
and satisfying the compatibility conditions \eqref{compat}, the system \eqref{Stvenant}, \eqref{Z}, \eqref{boundary1} has a unique solution $(H,V,Z)\in (C^{0}([0,T],H^{2}((0,L))))^{2}\times C^{1}([0,T])$. Moreover
there exists a positive constant $C(T)$ such that
\begin{equation}
\begin{split}
\lVert H(t,\cdot)-H_{1}(t,\cdot)\rVert_{H^{2}(0,L)}&+\lVert V(t,\cdot)-V_{1}(t,\cdot)\rVert_{H^{2}(0,L)}+|Z|\\
&\leq C(T)\left(\lVert H^{0}(\cdot)-H_{1}(0,\cdot)\rVert_{H^{2}(0,L)}+\lVert V^{0}(\cdot)-V_{1}(0,\cdot)\rVert_{H^{2}(0,L)}+|Z^{0}|\right).
\end{split}
\label{estimate}
\end{equation}
\label{th0}
\end{thm} 
To apply the result from \cite{Wang}, note that $Z$ can be seen as a third component of the hyperbolic system with a null propagation speed, a constant initial condition $Z^{0}$ and
$Z(t)$ being thus its value everywhere \textcolor{black}{on $[0,L]$} including at the boundaries.
\begin{rmk}
If, in addition, $(H^{0},V^{0})\in H^{3}((0,L);\mathbb{R}^{2})$, then the unique solution $(H,V,Z)$ given by Theorem \ref{th0} belongs to $C^{0}([0,T],H^{3}((0,L);\mathbb{R}^{2}))\times C^{2}([0,T])$ and there exists 
a constant $C(T)$ such that
\begin{equation}
\begin{split}
\lVert H(t,\cdot)-H_{1}(t,\cdot)\rVert_{H^{3}\textcolor{black}{(0,L)}}&+\lVert V(t,\cdot)-V_{1}(t,\cdot)\rVert_{H^{3}\textcolor{black}{(0,L)}}+|Z|\\
&\leq C(T)\left(\lVert H^{0}(\cdot)-H_{1}(0,\cdot)\rVert_{H^{3}\textcolor{black}{(0,L)}}+\lVert V^{0}(\cdot)-V_{1}(0,\cdot)\rVert_{H^{3}\textcolor{black}{(0,L)}}+|Z^{0}|\right).
\end{split}
\label{estimate2}
\end{equation}
\label{r0}
\end{rmk}
We recall the definition of exponential stability
\begin{defn}
We say that a trajectory $(H_{1},V_{1})$ is exponentially stable for the $H^{2}$ norm if there exists $\nu>0$, $C>0$ and $\gamma>0$ such that for any $T>t_{0}\geq 0$ and any $(H^{0},V^{0},Z^{0})$ satisfying
\begin{equation}
\lVert H^{0}(\cdot)-H_{1}(t_{0},\cdot)\rVert_{H^{2}(0,L)}+\lVert V^{0}(\cdot)-V_{1}(t_{0},\cdot)\rVert_{H^{2}(0,L)}+|Z^{0}|\leq \nu,
\end{equation}
and the compatibility conditions \eqref{compat}, the system \eqref{Stvenant}, \eqref{Z}, \eqref{boundary1} \textcolor{black}{with initial condition $(H^{0},V^{0},Z^{0})$ at $t_{0}$} has a unique solution $(H,V,Z)\in (C^{0}([t_{0},T],H^{2}((0,L))))^{2}\times C^{1}([t_{0},T])$ and,
\begin{equation}
\begin{split}
&\lVert H(t,\cdot)-H_{1}(t,\cdot)\rVert_{H^{2}(0,L)}+\lVert V(t,\cdot)-V_{1}(t,\cdot)\rVert_{H^{2}(0,L)}+|Z|\\
&\leq Ce^{-\gamma t}\left(\lVert H^{0}(\cdot)-H_{1}(t_{0},\cdot)\rVert_{H^{2}(0,L)}+\lVert V^{0}(\cdot)-V_{1}(t_{0},\cdot)\rVert_{H^{2}(0,L)}+|Z^{0}|\right),\text{  }\forall\text{  }t\in[t_{0},T].
\end{split}
\end{equation} 
\end{defn}
\begin{rmk}
From \eqref{target} and Sobolev inequality, this exponential stability implies in particular the exponential convergence of $H(t,L)$ to $H_{c}$.
\end{rmk}

\vspace{\baselineskip}
We can now state the main results of this article
\begin{thm}[Exponential stability]
There exists $\delta>0$ such that if $\lVert\partial_{t}Q_{0}\rVert_{C^{3}([0,+\infty))}\leq \delta$, then the trajectory $(H_{1},V_{1})$ given by \eqref{target} of system \eqref{Stvenant}, \eqref{Z}, \eqref{boundary1} is exponentially stable for the $H^{2}$ norm if:
\begin{equation}
\begin{split}
&k_{p}>-1\text{  and  }k_{I}>0,\\
\text{  or   }&k_{p}<-1-\frac{gH_{1}(t,L)-V^{2}_{1}(t,L)}{v_{G}V_{1}(t,L)}\text{   and   }k_{I}<0.
\end{split}
\label{cond}
\end{equation}
\label{th1}
\end{thm}
This result is proved in Section \ref{s2}. 
The main idea of the proof consist in finding a local convex and dissipative entropy for the system  \eqref{Stvenant}, \eqref{Z}, \eqref{boundary1}.\\

In particular, in the case where $Q_{0}$ is constant, we can use the static boundary control \eqref{defu1}, and we have the following corollary:
\begin{cor}
If $Q_{0}$ is constant, then the steady-state $(H^{*},V^{*})$ of the system \eqref{Stvenant}, \eqref{Z}, \eqref{boundary2} given by \eqref{targetsteady0}--\eqref{targetsteady}
is exponentially stable for the $H^{2}$ norm if:
\begin{equation}
\begin{split}
k_{p}>-1\text{  and  }&k_{I}>0,\\
\textcolor{black}{\text{  or   }k_{p}<-1-\frac{gH^{*}(L)-V^{*2}(L)}{v_{G}V^{*}(L)}}&\text{   and   }k_{I}<0.
\end{split}
\label{cond1}
\end{equation}
\label{cor1}
\end{cor} 
\begin{proof}
This is a particular case of Theorem \ref{th1}. To see this, note, as mentioned earlier, that when $Q_{0}$ is constant, then $(H_{1},V_{1})=(H^{*},V^{*})$. 
Then, observe that $f(t)$ given in \eqref{defu12} is a constant that can be added in $Z$ (i.e. we can re-define $Z:=Z-f(t)$, which still satisfies \eqref{Z}). 
\end{proof}

\subsection{Comparison with existing results and contribution of this paper}
Many results exists in the literature concerning this stabilization problem (e.g. \cite{dos2008, XuSallet, BastinCoron2013, XuSallet2014, Trinh2015, BastinCoronTamasoiu2015, BastinCoronPI}). To our knowledge the most advanced result for the full non-linear system is \cite{BastinCoronPI} where the authors shows that if there is no slope, i.e. $C(x) = 0$, then the system can always be stabilized by the PI control \eqref{defu1} as long as the steady-state exists, and they give the sufficient condition 
\begin{equation}\label{condbastincoron}
k_{p}>0\;\text{ and }\;k_{I}>0.
\end{equation}
Note that this is the first result that allows an arbitrary size of source term and length.
In this paper, using a different type of Lyapunov function, we manage to show a more general result. Our main contributions are the following:
\begin{itemize}
\item The result holds for an arbitrary friction and also an arbitrary slope $C(x)\in C^{2}([0,L])$. Physically this means that the source can be non-dissipative and increase the energy of the system compared to the case where there is only friction.
\item We find a less restrictive stability condition
\begin{equation}
k_{p}>-1\;\text{ and }\;k_{I}>0,
\end{equation}
and we also show that another condition is sufficient:
\begin{equation*}
k_{p}<-1-\frac{gH^{*}(L)-V^{*2}(L)}{v_{G}V^{*}(L)}\;\text{ and }\; k_{I}<0.
\end{equation*}
This one is counter intuitive as $k_{p}<-1$ and $k_{I}<0$. \textcolor{black}{This means that if the height of the water is too high at $L$ the control would reduce the aperture of the gate in $L$ and reduce the flow that we let exit the system,
which intuitively should increase even more the height of the water at $L$.}
\item Our result holds also when stabilizing a slowly varying trajectory rather than a steady-state (that might not exists in practical case). In this case we use a kind of feedforward term in the boundary control.
\item In addition to the exponential stability, we show the Input-to-State stability with respect to an unknown inflow. In this case the only knowledge required on the system is the height of the water at $x=L$.
\end{itemize}

\textcolor{black}{Note that, just like \cite{BastinCoronPI}, this approach uses very little knowledge of the state of the system, as we only measure the height at the boundary $x=L$.}

\begin{rmk}
\textcolor{black}{When the system is homogeneous, our conditions \eqref{cond1} are optimal (necessary and sufficient) \cite{BastinCoronTamasoiu2015}, \cite[Section 2.2.4.1]{BastinCoron1D}.}
\end{rmk}

\begin{rmk}[Alternative notation in literature]
In the literature, results about PI control of the Saint-Venant equations sometimes leave the step of modeling the spillway and use a generic formulation of the PI control on the outflow rate of the form
\begin{equation}
H(t,L)V(t,L)=k_{1}(H(t,L)-H_{c})-k_{2}Z,
\end{equation}
where $Z$ is the integral term, still given by \eqref{Z}.
Note that, with these notations, the sufficient condition of Corollary \ref{cor1} becomes 
\begin{equation}
\begin{split}
k_{p}>0\text{  and  }&k_{I}>0,\\
\textcolor{black}{\text{  or   }k_{p}<-\frac{gH^{*}(L)-V^{*2}(L)}{V^{*}(L)}}&\text{   and   }k_{I}<0.
\end{split}
\end{equation} 
\end{rmk}

\subsection{Case of time-varying input disturbance $Q_{0}(t)$: ISS estimate}
In practical situation, however, we may have also 
little knowledge of the target trajectory $(H_{1},V_{1})$ or the input disturbance $Q_{0}(t)$ and we only know $H_{c}$. In this case we cannot use a controller of the form \eqref{boundary1},
but only a static controller of the form \eqref{boundary2}, namely
\begin{equation}
H(t,L)V(t,L)=v_{G}(1+k_{p})H(t,L)-v_{G}k_{p}H_{c}-v_{G}k_{I}Z.
\label{boundary22}
\end{equation} 
In this case, it is impossible to aim at stabilizing the target trajectory $(H_{1},V_{1})$, but we still have the Input-to-state Stability with respect to the input disturbance $\partial_{t}Q_{0}$,
\begin{thm}
There exists $\nu>0$, $\delta>0$, $\gamma>0$ and $C$, such that if
$\lVert\partial_{t}Q_{0}\rVert_{C^{2}([0,+\infty))}\leq\delta$, then for any $T>0$ and $(H^{0},V^{0})\in \textcolor{black}{(H^{2}(0,L))^{2}}$ such that
\begin{equation*}
\lVert H^{0}-H^{*}\rVert_{H^{2}(0,L)}+\lVert V^{0}-V^{*}\rVert_{H^{2}(0,L)}\leq \nu,
\end{equation*} 
the system \eqref{Stvenant}, \eqref{Z}, \eqref{boundary2} with initial condition $(H^{0},V^{0})$ has a unique solution $(H,V)\in C^{0}([0,T],H^{2}(0,L))$ which satisfies the following ISS inequality
\begin{equation}
\begin{split}
&\lVert H(t,\cdot)-H_{0}(t,\cdot)\rVert_{H^{2}(0,L)}+\lVert V(t,\cdot)-V_{0}(t,\cdot)\rVert_{H^{2}(0,L)}\\
&\leq \textcolor{black}{Ce^{-\gamma t}\left(\lVert H^{0}-H^{*},V^{0}-V^{*}\rVert_{H^{2}(0,L)}+\int_{0}^{t}(|\partial_{t}Q_{0}(s)|+|\partial_{tt}^{2}Q_{0}(s)|+|\partial_{ttt}^{3}Q_{0}(s)|)e^{\gamma s}ds\right)}\textcolor{black}{.}
\end{split}
\label{ISSHV}
\end{equation}
\label{th2}
\end{thm} 
The proof is given in Appendix \ref{ISS2} and is a consequence from the proof of Theorem \ref{th1}.\\

In Section \ref{s2} we prove Theorem \ref{th1}.
\section{Exponential stability for the $H^{2}$ norm}
\label{s2}
This section is divided in three parts. First we transform the system through a change of variables. 
Then we state two lemma, which simplify the proof of Theorem \ref{th1}. 
Finally we prove Theorem \ref{th1}.

\subsection{A change of variables}
For any solution of \eqref{Stvenant}, \eqref{Z}, \eqref{boundary1} 
we define the perturbation as
\begin{equation}
\begin{pmatrix}
h\\v
\end{pmatrix}=
\begin{pmatrix}
H-H_{1}\\
V-V_{1}
\end{pmatrix}.
\label{change1}
\end{equation}
Let us assume that there exists $\nu\in(0,\nu_{0})$ to be selected later on, such that 
\begin{equation}
\lVert H^{0}(\cdot)-H_{1}(0,\cdot)\rVert_{H^{2}(0,L)}+\lVert V^{0}(\cdot)-V_{1}(0,\cdot)\rVert_{H^{2}(0,L)}+|Z^{0}|\leq \nu.
\end{equation}
The boundary conditions \eqref{boundary1} can be written in the following form
\begin{equation}
\begin{split}
&v(t,0)=\mathcal{B}_{1}(h(t,0),t),\\
&v(t,L)=\mathcal{B}_{2}(h(t,L),Z,t),
\end{split}
\label{bound1}
\end{equation} 
with
\begin{equation}
\begin{split}
&\partial_{1}\mathcal{B}_{1}(0,t)=-\frac{V_{1}(t,0)}{H_{1}(t,0)},\\
&\partial_{1}\mathcal{B}_{2}(0,0,t)=\frac{v_{G}(1+k_{p})-V_{1}(t,L)}{H_{1}(t,L)},\\
&\partial_{2}\mathcal{B}_{2}(0,0,t)=-\frac{v_{G}k_{I}}{H_{1}(t,L)}.
\end{split}
\label{k1k30}
\end{equation}
We introduce the following change of variables:
\begin{equation}
\mathbf{u}: =\begin{pmatrix}
u_{1}\\u_{2}
\end{pmatrix}=
\begin{pmatrix}
v+\sqrt{\frac{g}{H_{1}}}h\\v-\sqrt{\frac{g}{H_{1}}}h
\end{pmatrix}.
\label{change2}
\end{equation} 
Note that this change of variables is very similar to the change of variables used in \cite{BastinCoron22,HS} with the only difference that $(H_{1},V_{1})$ is not a steady-state anymore. 
It corresponds to the transformation in Riemann coordinates for the perturbations. 
Indeed, denoting $S$, $F$ and $G$ by
\begin{gather}
S(x,t)=\begin{pmatrix}\sqrt{\frac{g}{H_{1}(t,x)}} & 1\\-\sqrt{\frac{g}{H_{1}(t,x)}} & 1\end{pmatrix},
\label{defS}\\
F\begin{pmatrix}H\\V\end{pmatrix}=\begin{pmatrix}
        V & H\\
        g & V
       \end{pmatrix},
       \label{defF}\text{  }\text{  }\text{  }\text{  }\text{  }
       G\begin{pmatrix}H\\V\end{pmatrix}=\begin{pmatrix}0\\\frac{kV^{2}}{H}-C(x)\end{pmatrix}\textcolor{black}{,}
\end{gather}
and using \eqref{Stvenant}, \eqref{Z}, \eqref{boundary1}, \eqref{target}, \eqref{change1}--\eqref{change2}, one has
\begin{equation}
\begin{split}
\partial_{t}u_{1}+\Lambda_{1}(\mathbf{u},x,t)\partial_{x}u_{1}+l_{1}(\mathbf{u},x,t)\partial_{x}u_{2}+B_{1}(\mathbf{u},x,t)=0\textcolor{black}{,}\\
\partial_{t}u_{1}-\Lambda_{2}(\mathbf{u},x,t)\partial_{x}u_{2}+l_{2}(\mathbf{u},x,t)\partial_{x}u_{1}+B_{2}(\mathbf{u},x,t)=0\textcolor{black}{,}
\end{split}
\label{sys1}
\end{equation} 
where, 
\begin{equation}
\begin{pmatrix}\Lambda_{1}(\mathbf{u},x,t) & l_{1}(\mathbf{u},x,t)\\l_{2}(\mathbf{u},x,t) & \Lambda_{2}(\mathbf{u},x,t)\end{pmatrix}=S(x,t)F\left(S^{-1}(x,t)\mathbf{u}+\begin{pmatrix}H_{1}(t,x)\\V_{1}(t,x)\end{pmatrix}\right)S^{-1}(x,t)=:A(\mathbf{u},x,t),
\label{defA}
\end{equation}
and
\begin{equation}
\begin{split}
B(\mathbf{u},x,t)=&\begin{pmatrix}B_{1}(\mathbf{u},x,t)\\B_{2}(\mathbf{u},x,t)\end{pmatrix}
=S(x,t)F\left(S^{-1}(x,t)\mathbf{u}+\begin{pmatrix}H_{1}(t,x)\\V_{1}(t,x)\end{pmatrix}\right)\left(\begin{pmatrix} \partial_{x}H_{1}(t,x)\\ \partial_{x}V_{1}(t,x)\end{pmatrix}+\partial_{x}(S^{-1})\mathbf{u}\right)\\
&+S\partial_{t}\begin{pmatrix}H_{1}(t,x)\\V_{1}(t,x)\end{pmatrix}+S(x,t)G\left(S^{-1}(x,t)\mathbf{u}+\begin{pmatrix}H_{1}(t,x)\\V_{1}(t,x)\end{pmatrix}\right)-\partial_{t}S(x,t)S^{-1}(x,t)\mathbf{u}.
\end{split}
\label{defB}
\end{equation}
Therefore
\begin{gather}
\Lambda_{1}(0,x,t)=V_{1}+\sqrt{gH_{1}},\text{  }\text{  }\text{  }\Lambda_{2}(0,x,t)=V_{1}-\sqrt{gH_{1}}\label{Lambda}\textcolor{black}{,}\\
l_{1}(0,x,t)=B_{1}(\mathbf{0},x,t)=0,\text{  }\text{  }\text{  }l_{2}(0,x,t)=B_{2}(\mathbf{0},x,t)=0\textcolor{black}{,}\\
\begin{split}
\frac{\partial{B}_{1}}{\partial_{u}}(0,x,t)&=\gamma_{1}(t,x)u_{1}(t,x)+\gamma_{2}(t,x)u_{2}(t,x),\\
\frac{\partial{B}_{2}}{\partial_{u}}(0,x,t)&=\delta_{1}(t,x)u_{1}(t,x)+\delta_{2}(t,x)u_{2}(t,x).
\end{split}
\label{l1l2}
\end{gather}
where
\begin{equation}
\begin{split}
\gamma_{1}&=\frac{3}{4}\sqrt{\frac{g}{H_{1}}}H_{1x}+\frac{3}{4}V_{1x}+\frac{kV_{1}}{H_{1}}-\frac{kV_{1}^{2}}{2H_{1}^{2}}\sqrt{\frac{H_{1}}{g}}\\
\gamma_{2}&=\frac{1}{4}\sqrt{\frac{g}{H_{1}}}H_{1x}+\frac{1}{4}V_{1x}+\frac{kV_{1}}{H_{1}}+\frac{kV_{1}^{2}}{2H_{1}^{2}}\sqrt{\frac{H_{1}}{g}}\\
\delta_{1}&=-\frac{1}{4}\sqrt{\frac{g}{H_{1}}}H_{1x}+\frac{1}{4}V_{1x}+\frac{kV_{1}}{H_{1}}-\frac{kV_{1}^{2}}{2H_{1}^{2}}\sqrt{\frac{H_{1}}{g}}\\
\delta_{2}&=-\frac{3}{4}\sqrt{\frac{g}{H_{1}}}H_{1x}+\frac{3}{4}V_{1x}+\frac{kV_{1}}{H_{1}}+\frac{kV_{1}^{2}}{2H_{1}^{2}}\sqrt{\frac{H_{1}}{g}}.
\end{split}
\label{gamma}
\end{equation}
And for the boundary conditions, there exists $\nu_{1}\in(0,\nu_{0})$ such that for any $\nu\in(0,\nu_{1})$, one has:
\begin{equation}
\begin{split}
&u_{1}(t,0)=\mathcal{D}_{1}(u_{2}(t,0),t),\\
&u_{2}(t,L)=\mathcal{D}_{2}(u_{1}(t,L),Z,t),\\
&\dot Z=\frac{(u_{1}(t,L)-u_{2}(t,L))}{2}\sqrt{\frac{H_{1}(t,L)}{g}}\textcolor{black}{,}
\end{split}
 \label{bound}
\end{equation}
where $\mathcal{D}_{1}$ and $\mathcal{D}_{2}$ are $C^{2}$ functions and
\begin{equation}
\begin{split}
&
\partial_{1}\mathcal{D}_{1}(0,t)=-\frac{\lambda_{2}(0)}{\lambda_{1}(0)},\\
&
\partial_{1}\mathcal{D}_{2}(0,0,t)=-\frac{\lambda_{1}(L)-v_{G}(1+k_{p})}{\lambda_{2}(L)+v_{G}(1+k_{p})},\\
&
\partial_{2}\mathcal{D}_{2}(0,0,t)=-2\frac{v_{G}k_{I}\sqrt{\frac{g}{H_{1}(t,L)}}}{v_{G}(1+k_{p})+\lambda_{2}(t,L)}.
\end{split}
\label{k1k3}
\end{equation}
Expression \eqref{gamma} is simply a computation, very similar to what it done in \cite{HS} for instance, while the derivation of \eqref{bound} and \eqref{k1k3} are detailed in the appendix.
\begin{rmk}
Obviously, from the change of variables \eqref{change1}--\eqref{change2}, the exponential stability of 
the system \eqref{Stvenant}, \eqref{Z}, \eqref{boundary1} 
is equivalent to the exponential stability of the steady-state $\mathbf{u}^{*}=0$ for the system \eqref{sys1}, \eqref{bound}.
\label{r1}
\end{rmk}
As the operator $A$, given by \eqref{defA}, is a \textcolor{black}{$C^{2}$ function in $\mathbf{u}$,
$t$ and $x$ (and \textcolor{black}{in particular $C^{1}$})}
and as, from \eqref{l1l2} and \eqref{fluvial}, $\Lambda_{1}(\mathbf{0},x,t)>0>\Lambda_{2}(\mathbf{0},x,t)$,
there exists $\nu_{2}\in(0,\nu_{1})$ and $E\in C^{1}(\mathcal{B}_{\nu_{2}}\times(0,L)\times[0,+\infty);\mathcal{M}_{2}(\mathbb{R}))$, where 
$\mathcal{B}_{\nu_{2}}\subset\mathbb{R}^{2}$ is the disc of radius $\nu_{2}$ and center $0$,
such that
for any $\lVert\mathbf{u}(t,\cdot)\rVert_{H^{2}(0,L)}\leq \nu_{2}$,
\begin{equation}
\begin{split}
E(\mathbf{u}(t,x),x,t)A(\mathbf{u}(t,x),x,t)&=D(\mathbf{u}(t,x),x,t)E(\mathbf{u}(t,x),x,t)\textcolor{black}{,}\\
E(\mathbf{0},x,t)&=Id\textcolor{black}{,}
\end{split}
\label{definv}
\end{equation} 
where $D(\mathbf{u}(t,x),x,t)=(D_{i}(\mathbf{u}(t,x),x,t))_{i\in{1,2}}$ is a diagonal matrix and $Id$ is the identity matrix.
Before going any further, let us 
note a few useful properties of these functions. \textcolor{black}{For simplicity in the following we will denote for any $n\in\mathbb{N}^{*}$ and any function $U\in L^{\infty}((0,T)\times(0,L);\mathbb{R}^{n})$ (resp. $L^{\infty}((0,L);\mathbb{R}^{n})$)
\begin{equation}
\begin{split}
&\lVert U\rVert_{\infty}:=\lVert U\rVert_{L^{\infty}((0,T)\times(0,L);\mathbb{R}^{n})},\\
&\text{(resp.}\lVert U\rVert_{\infty}:=\lVert U\rVert_{L^{\infty}((0,L);\mathbb{R}^{n})}\text{)}.
\end{split}
\end{equation} We may also denote $\lVert\mathbf{u}\rVert_{H^{2}(0,L)}$ instead of $\lVert\mathbf{u}(t,\cdot)\rVert_{H^{2}(0,L)}$ to lighten the expressions.}
From the definition of $A$ given in \eqref{defA},
and from \eqref{fluvial},
for $\lVert\mathbf{u}\rVert_{H^{2}(0,L)}\leq \nu_{2}$, there exists a constant $C_{1}$ depending only on $H_{\max}$, $\alpha$ and $\nu_{2}$ such that we have the following estimates
\begin{equation}
\begin{split}
&\max\left(\lVert \partial_{t}(A(\mathbf{u}(t,x),x,t)-A(\mathbf{0},x,t))\rVert_{\infty},\lVert \partial_{t}(D(\mathbf{u}(t,x),x,t)-D(\mathbf{0},x,t))\rVert_{\infty},\lVert \partial_{t}(E(\mathbf{u}(t,x),x,t))\rVert_{\infty}\right)\\
&\leq C_{1}\left(\lVert\mathbf{u}\rVert_{\infty}(\lVert\partial_{t}H_{1}\rVert_{\infty}+\lVert\partial_{t}V_{1}\rVert_{\infty})+\lVert\partial_{t}\mathbf{u}\rVert_{\infty}\right),\\
&\max\left(\lVert \partial_{t}(A(\mathbf{u}(t,x),x,t)-A(\mathbf{0},x,t))\rVert_{\infty},\lVert \partial_{t}(D(\mathbf{u}(t,x),x,t)-D(\mathbf{0},x,t)),\lVert \partial_{t}(E(\mathbf{u}(t,x),x,t))\rVert_{\infty}\right)\\
&\leq C_{1}\left(\lVert\mathbf{u}\rVert_{\infty}(\lVert\partial_{x}H_{1}\rVert_{\infty}+\lVert\partial_{x}V_{1}\rVert_{\infty})+\lVert\partial_{x}\mathbf{u}\rVert_{\infty}\right).
\end{split}
\label{ineq}
\end{equation} 
For $E$ and $D$, this comes from the fact that $E$ and $D$ are $C^{\infty}$ functions with respect to the coefficients of $A$ (note that $D$ is the matrix of eigenvalues of $A$), and that $A\in C^{2}(\mathcal{B}_{\eta_{0}}; C^{1}([0,+\infty)\times[0,L]))$.
\vspace{\baselineskip}

\subsection{Two useful lemma}
We introduce now two lemma, which simplify the proof of Theorem \ref{th1} The first one is a classical result about Lyapunov functions,
\begin{lem}
Let $V: (H^{2}(0,L))^{2}\times\mathbb{R}\times\mathbb{R}_{+}\rightarrow \mathbb{R}_{+}^{*}$ such that there exists a constant $c>0$ such that
\begin{equation}
c\left(\lVert \mathbf{U} \rVert_{H^{2}(0,L)}+|z|\right)\leq V(\mathbf{U},z,t)\leq \frac{1}{c} \left(\lVert \mathbf{U} \rVert_{H^{2}(0,L)}+|z|\right),\text{   }\forall\text{   }(U,z,t)\in (H^{2}(0,L))^{2}\times\mathbb{R}\times\mathbb{R}_{+}.
\label{cV1}
\end{equation}
If there exists $\gamma>0$ and $\delta>0$ such that, for any solution
$(\mathbf{u},Z)$ of the system \eqref{sys1}, \eqref{bound} with initial conditions
satisfying $\lVert \mathbf{u}(0,\cdot)\rVert_{H^{2}(0,L)}+|Z(0)|\leq \delta$, 
\begin{equation}
\frac{d}{dt}\left[V(\mathbf{u}(t,\cdot),t)\right]<-\gamma V(\mathbf{u}(t,\cdot),t)
\label{cV2}
\end{equation}
in a distribution sense, then the system \eqref{sys1}, \eqref{bound} is exponentially stable for the $H^{2}$ norm and $V$ is called a Lyapunov function for the system \eqref{sys1}, \eqref{bound}.
\label{lem1}
\end{lem}
This first lemma reduces the problem of proving the exponential stability to finding a Lyapunov function $V$ for the system \eqref{sys1}, \eqref{bound}. A proper definition of a differential inequality in a distribution sense as in \eqref{cV2} can be found in \cite{C1}. To lighten this article we do not give a proof of this classical lemma, although a proof for a very similar case (Lyapunov function that does not depend explicitly on time and for the $C^{1}$ norm instead) can be found for instance in \cite{C1}[Proposition 2.1], and is easily extended to this case. \\

Our second lemma seems very natural:
\begin{lem}
There exists $l>0$ and $C>0$ 
such that if $\lVert\partial_{t}Q_{0}\rVert_{C^{3}([0,+\infty))}\leq l$, then
\begin{equation}
\begin{split}
&\max\left(\lVert\partial_{t}H_{1}\rVert_{C^{1}([0,+\infty), \textcolor{black}{C^{0}([0,L])})},\lVert\partial_{t}V_{1}\rVert_{C^{1}([0,+\infty), \textcolor{black}{C^{0}([0,L])})}\right)< C \lVert\partial_{t}Q_{0}\rVert_{C^{3}([0,+\infty))}.
\end{split}
\end{equation} 
\label{lem3}
\end{lem} 
This is a consequence of the ISS property (Proposition \ref{propISS}) and Remark \ref{rmkHpISS} for $p=3$ and is shown in Appendix \ref{app:prooflem3}. 
Thanks to this Lemma, we now only need to show Theorem \ref{th1} with a bound on  $\partial_{t}H_{1}$ and $\partial_{t}V_{1}$ rather than a bound on $\partial_{t}Q_{0}$.
\subsection{Proof of Theorem \ref{th1}}
We can now prove Theorem \ref{th1}.
\begin{proof}[Proof of Theorem \ref{th1}]
From Theorem \eqref{th0}, Remark \ref{r1}, and Lemma \ref{lem1}, one only needs to find a Lyapunov function $V: (H^{2}(0,L))^{2}\times\mathbb{R}\times\mathbb{R}_{+}\rightarrow \mathbb{R}_{+}^{*}$ satisfying 
\eqref{cV1} and \eqref{cV2}. We define the following candidate:
\begin{equation}
\begin{split}
V_{a}(\mathbf{U},z,t):=&\int_{0}^{L}f_{1}(t,x)e^{-\mu x}(E(\mathbf{U}(x),x,t)\mathbf{U})_{1}^{2}(t,x)\\
&+f_{2}(t,x)e^{\mu x}(E(\mathbf{U}(x),x,t)\mathbf{U})_{2}^{2}(t,x)dx+qz^{2}\textcolor{black}{,}
\end{split}
\label{Va}
\end{equation}
where $f_{1}$, $f_{2}$ are positive and bounded functions which will be defined later on, and $\mu$ and $q$ are positives constant which will also be defined later on. Recall that $E$ is still given by \eqref{definv}.
Let $T>\textcolor{black}{t_{0}\geq0}$ and $(\mathbf{u}^{0},Z^{0})\in H^{2}(0,L)\times\mathbb{R}$ satisfying the compatibility condition \eqref{compat} and such that 
\begin{equation}
\left(\lVert \mathbf{u}^{0}\rVert_{H^{2}(0,L)}+|Z^{0}|\right)<\nu,
\label{boundu}
\end{equation} 
where $\nu$ is a constant to be chosen later on but such that $\nu<\min(\nu_{2},\nu(T))$. Recall that $\nu(T)$ is given by Theorem \ref{th0}. 
From Theorem \ref{th0} there exists a unique solution $\mathbf{u}\in C^{0}([\textcolor{black}{t_{0}},T],H^{2}(0,L))$.
We suppose in addition that $(\mathbf{u}^{0},Z^{0})\in H^{3}(0,L)$, and that \eqref{boundu} also hold for the $H^{3}$ norm instead of the $H^{2}$ norm in $u$. From Remark \ref{r0}, 
$(\mathbf{u},Z)\in C^{0}([\textcolor{black}{t_{0}},T]\times H^{3}(0,L))\times C^{3}([\textcolor{black}{t_{0}},T])$.
This assumption is here to allow us to compute easily the derivative of $\mathbf{u}$ but will be relaxed later on by density.\\

Let $\delta>0$ to be chosen later on,
 and assume that 
\begin{equation}
\max(\lVert\partial_{t}H_{1}\rVert_{C^{1}([t_{0},\infty);\textcolor{black}{C^{0}([0,L])})},\lVert\partial_{t}V_{1}\rVert_{C^{1}([t_{0},\infty);\textcolor{black}{C^{0}([0,L])})}<\delta.
\label{boundHV}
\end{equation} 
\textcolor{black}{As this is the only assumption on $H_{1}$ and $V_{1}$, we can assume from now on that $t_{0}=0$ without loss of generality.}\\

Looking at \eqref{Va}, $V_{a}$ is indeed a function defined on $H^{2}(0,L)\times\mathbb{R}\times\mathbb{R}_{+}$, but for notational ease we will denote \textcolor{black}{$V_{a}(t):=V_{a}(\mathbf{u}(t,\cdot),Z(t),t)$, where 
$Z(t)$ is given by \eqref{Z},}
and $E:=E(\mathbf{u}(t,x),x,t)$.
Similarly we introduce
\begin{equation}
\begin{split}
V_{b}(\mathbf{U},t):=&\int_{0}^{L}f_{1}e^{-\mu x}(E(\mathbf{U}(x),x,t)\mathbf{I}(\mathbf{U},x,t))^{2}_{1}+f_{2}e^{\mu x}(E(\mathbf{U}(x),x,t)\mathbf{I}(\mathbf{U},x,t))^{2}_{2}dx\\
&+q\frac{H_{1}(t,L)}{4g}(U_{1}(L)-U_{2}(L))^{2}\textcolor{black}{,}\\
V_{c}(\mathbf{U},t):=&\int_{0}^{L}f_{1}e^{-\mu x}(E(\mathbf{U}(x),x,t)\mathbf{J}(\mathbf{U},x,t))^{2}_{1}+f_{2}e^{\mu x}(E(\mathbf{U}(x),x,t)\mathbf{J}(\mathbf{U},x,t))^{2}_{2}dx\\
&+q\left(\sqrt{\frac{H_{1}(t,L)}{4g}}(I_{1}(t,L)-I_{2}(t,L))+\frac{\partial_{t}H_{1}(t,L)}{4}\sqrt{\frac{1}{gH_{1}(t,L)}}(U_{1}(L)-U_{2}(L))\right)^{2}\textcolor{black}{,}
\label{V2V3}
\end{split}
\end{equation} 
where 
\begin{equation}
\begin{split}
\mathbf{I}(\mathbf{U},x,t):=&\left(A(\mathbf{U},x,t)\partial_{x}(\partial_{t}\mathbf{U})+(\partial_{t}A(\mathbf{U},x,t)+\partial_{\mathbf{U}}A(\mathbf{U},x,t).\partial_{t}\mathbf{U})\partial_{x}\mathbf{U}+\partial_{t}B\left(\mathbf{U},x,t\right)\right.\\
&+\left.(\partial_{\mathbf{U}}B(\mathbf{U},x,t))(\partial_{t}\mathbf{U})\right)\textcolor{black}{,}\\
\mathbf{J}(\mathbf{U},x,t):=&A(\mathbf{U},x,t)\partial_{x}(\partial_{tt}^{2}\mathbf{U})
+(\partial_{\mathbf{U}}A(\mathbf{U},x).\partial_{tt}^{2}\mathbf{U})\partial_{x}\mathbf{U}
+(\partial_{\mathbf{U}}B(\mathbf{U},x))(\partial_{tt}^{2}\mathbf{U})\\
&+(\partial_{tt}^{2}A(\mathbf{U},x,t)+2\partial_{\mathbf{U}}(\partial_{t}A(\mathbf{U},x,t)).\partial_{t}\mathbf{U})\partial_{x}\mathbf{U}\\
&+2\partial_{t}A(\mathbf{U},x,t)\partial_{x}(\partial_{t}\mathbf{U})+2\partial_{\mathbf{U}}A(\mathbf{U},x).\partial_{t}\mathbf{U}\partial_{x}(\partial_{t}\mathbf{U})
\\
&+((\partial^{2}_{\mathbf{U}}A(\mathbf{U},x).\partial_{t}\mathbf{U}).\partial_{t}\mathbf{U})\partial_{x}\mathbf{U}\\
&+\partial_{tt}^{2}B(\mathbf{U},x)+2\partial_{\mathbf{U}}(\partial_{t}B(\mathbf{U},x)).\partial_{t}\mathbf{U}
+(\partial^{2}_{\mathbf{U}}B(\mathbf{U},x).\partial_{t}\mathbf{U})(\partial_{t}\mathbf{U})\textcolor{black}{.}
\end{split}
\label{IJ}
\end{equation}
Observe that for a solution $\mathbf{u}$ of \eqref{sys1}, \textcolor{black}{and using the expression of $Z$ given by \eqref{Z}}, the expressions of $V_{b}(\mathbf{u}(t,\cdot),t)$ and $V_{c}(\mathbf{u}(t,\cdot),t)$ become
\begin{equation}
\begin{split}
V_{b}(\mathbf{u}(t,\cdot),t):=\int_{0}^{L}f_{1}(t,x)e^{-\mu x}(E\partial_{t}\mathbf{u})_{1}^{2}(t,x)+f_{2}(t,x)e^{\mu x}(E\partial_{t}\mathbf{u})_{2}^{2}(t,x)dx+q(\dot Z(t))^{2}\textcolor{black}{,}\\
V_{c}(\mathbf{u}(t,\cdot),t):=\int_{0}^{L}f_{1}(t,x)e^{-\mu x}(E\partial_{tt}^{2}\mathbf{u})_{1}^{2}(t,x)+f_{2}(t,x)e^{\mu x}(E\partial_{tt}^{2}\mathbf{u})_{2}^{2}(t,x)dx+q(\ddot Z(t))^{2}\textcolor{black}{,}
\label{V2V31}
\end{split}
\end{equation} 
which justifies the expression chosen for \eqref{V2V3} and \eqref{IJ}. We also note for notational ease $V_{b}(t):=V_{b}(t,\mathbf{u}(t,\cdot))$ and $V_{c}(t):=V_{c}(t,\mathbf{u}(t,\cdot))$.
Finally we denote $V:=V_{a}+V_{b}+V_{c}$. We start now by dealing with $V_{a}$,
Differentiating $t\rightarrow V_{a}(t)$ with respect to time, using \eqref{sys1}, \eqref{definv} and integrating by parts, one has
\begin{equation}
\begin{split}
 \dot V_{a}=& -2\int_{0}^{L}f_{1}(t,x)e^{-\mu x}(E\mathbf{u})_{1}\left[(EA(\mathbf{u},x,t)\partial_{x} \mathbf{u})_{1}+(EB)_{1}(\mathbf{u},x,t)\right]\\
 &+f_{2}(t,x)e^{\mu x}(E\mathbf{u})_{2}\left[(EA(\mathbf{u},x,t)\partial_{x} \mathbf{u})_{2}+(EB)_{2}(\mathbf{u},x,t)\right]dx\\
 &+\int_{0}^{L}\partial_{t}(f_{1})e^{-\mu x}(E\mathbf{u})_{1}^{2}+\partial_{t}(f_{2})e^{\mu x}(E\mathbf{u})_{2}^{2}dx
\\
&+2\int_{0}^{L}f_{1}e^{-\mu x}(E\mathbf{u})_{1}\left(\left(\partial_{t}E+\partial_{\mathbf{u}}E.\partial_{t}\mathbf{u}\right)\mathbf{u}\right)_{1}+f_{2}e^{\mu x}\left(\left(\partial_{t}E+\partial_{\mathbf{u}}E.\partial_{t}\mathbf{u}\right)\mathbf{u}\right)_{2}dx\\
& +2qZ(t)\dot Z(t)\\
 =&-2\int_{0}^{L}f_{1}(t,x)e^{-\mu x}(E\mathbf{u})_{1}\left[D_{1}(\mathbf{u},x,t)\left(\partial_{x} (E\mathbf{u})-(\partial_{x}E+\partial_{\mathbf{U}}E.\partial_{x}\mathbf{u})\mathbf{u}\right)_{1}\right]\\
 &+f_{2}(t,x)e^{\mu x}(E\mathbf{u})_{2}\left[D_{2}(\mathbf{u},x,t)\left(\partial_{x} (E\mathbf{u})-(\partial_{x}E+\partial_{\mathbf{U}}E.\partial_{x}\mathbf{u})\mathbf{u}\right)_{2}\right]dx\\
 &+\int_{0}^{L}\partial_{t}(f_{1})e^{-\mu x}(E\mathbf{u})_{1}^{2}+\partial_{t}(f_{2})e^{\mu x}(E\mathbf{u})_{2}^{2}dx\\
 &-2\int_{0}^{L}f_{1}e^{-\mu x}(E\mathbf{u})_{1}(EB)_{1}(\mathbf{u},x,t)+f_{2}e^{\mu x}(E\mathbf{u})_{2}(EB)_{2}(\mathbf{u},x,t)dx\\
&+2\int_{0}^{L}f_{1}e^{-\mu x}(E\mathbf{u})_{1}\left(\left(\partial_{t}E+\partial_{\mathbf{u}}E.\partial_{t}\mathbf{u}\right)\mathbf{u}\right)_{1}+f_{2}e^{\mu x}\left(\left(\partial_{t}E+\partial_{\mathbf{u}}E.\partial_{t}\mathbf{u}\right)\mathbf{u}\right)_{2}dx\\
 &+2qZ(t)\dot Z(t)\textcolor{black}{,}\\
 \end{split}
 \label{diffVa0}
 \end{equation}
\begin{equation}
\begin{split}
 \dot V_{a}=& -\left[f_{1}e^{-\mu x}D_{1}(E\mathbf{u})_{1}^{2}+D_{2}f_{2}e^{\mu x}(E\mathbf{u})_{2}^{2}\right]_{0}^{L}\\
 &-\int_{0}^{L}(E\mathbf{u})_{1}e^{-\mu x}\left((-\partial_{x}(D_{1}f_{1})-f_{1}\partial_{u}(D_{1}).\partial_{x}\mathbf{u})(E\mathbf{u})_{1}-2f_{1}D_{1}((\partial_{x}E+\partial_{\mathbf{U}}E.\partial_{x}\mathbf{u})\mathbf{u})_{1}\right)\\
&+(E\mathbf{u})_{2}e^{\mu x}\left((-\partial_{x}(D_{2}f_{2})-f_{2}\partial_{u}(D_{2}).\partial_{x}\mathbf{u})(E\mathbf{u})_{2}-2f_{2}D_{2}((\partial_{x}E+\partial_{\mathbf{U}}E.\partial_{x}\mathbf{u})\mathbf{u})_{2}\right)dx\\
&+\int_{0}^{L}\partial_{t}(f_{1})e^{-\mu x}(E\mathbf{u})_{1}^{2}+\partial_{t}(f_{2})e^{\mu x}(E\mathbf{u})_{2}^{2}dx\\
&-2\int_{0}^{L}f_{1}e^{-\mu x}(E\mathbf{u})_{1}(EB)_{1}(\mathbf{u},x,t)+f_{2}e^{\mu x}(E\mathbf{u})_{2}(EB)_{2}(\mathbf{u},x,t)\\
&+2\int_{0}^{L}f_{1}e^{-\mu x}(E\mathbf{u})_{1}\left(\left(\partial_{t}E+\partial_{\mathbf{u}}E.\partial_{t}\mathbf{u}\right)\mathbf{u}\right)_{1}+f_{2}e^{\mu x}\left(\left(\partial_{t}E+\partial_{\mathbf{u}}E.\partial_{t}\mathbf{u}\right)\mathbf{u}\right)_{2}dx\\
&-\mu\int_{0}^{L}D_{1}f_{1}e^{-\mu x}(E\mathbf{u})_{1}^{2}-D_{2}f_{2}e^{\mu x}(E\mathbf{u})_{2}^{2}dx
 +2qZ(t)\dot Z(t).\\
\end{split}
\label{diffVa2}
\end{equation}
In order to simplify this expression, observe that from \eqref{defA}, \eqref{definv} and \eqref{lambda}, $D_{1}(\mathbf{0},x,t)=\lambda_{1}(t,x)$ and $D_{2}(\mathbf{0},x,t)=-\lambda_{2}(t,x)$. Recall that $H_{1}$ is bounded by $H_{\max}$ from \eqref{Hinfty} and that $gH_{1}-V_{1}^{2}$ is bounded by below by $\alpha$ from \eqref{fluvial}. Using this, the fact that
$D$ is $C^{1}$ in $\mathbf{u}$,
and using also \eqref{ineq} and \eqref{boundHV}, there exists $C>0$ depending only on $H_{\max}$ and $\alpha$, $\nu$ and $\delta$ such that
\begin{equation}
\lVert D_{i}-\text{sgn}(D_{i})\lambda_{i}\rVert_{\infty}\leq \lVert C\mathbf{u}\rVert_{\infty},
\end{equation} 
\begin{equation}
\lVert\partial_{x}D_{i}+\partial_{\mathbf{u}}D_{i}.\partial_{x}\mathbf{u}-\text{sgn}(D_{i})\partial_{x}\lambda_{i}\rVert_{\infty}\leq C\left(\lVert \partial_{x}\mathbf{u}\rVert_{\infty}+\lVert \mathbf{u}\rVert_{\infty}\right),\text{  }i\in\{1,2\},
\end{equation} 
and
\begin{equation}
\lVert\partial_{x}E\rVert_{\infty}\leq C\left(\lVert \mathbf{u}\rVert_{\infty}\right),
\end{equation} 
\begin{equation}
\lVert\partial_{t}E+\partial_{\mathbf{u}}E\partial_{t}\mathbf{u}\rVert_{\infty}\leq C\left(\lVert \mathbf{u}\rVert_{\infty}+\lVert \partial_{t} \mathbf{u}\rVert_{\infty}\right).
\end{equation} 
Thus, using this together with \eqref{diffVa2}
\begin{equation}
\begin{split}
 \dot V_{a}\leq& -\left[f_{1}e^{-\mu x}D_{1}(E\mathbf{u})_{1}^{2}+D_{2}f_{2}e^{\mu x}(E\mathbf{u})_{2}^{2}\right]_{0}^{L}\\
 &-\int_{0}^{L}(E\mathbf{u})_{1}^{2}e^{-\mu x}(-\partial_{x}(\lambda_{1}f_{1})-\partial_{t}(f_{1}))+(E\mathbf{u})_{2}^{2}e^{\mu x}(\partial_{x}(\lambda_{2}f_{2})-\partial_{t}(f_{2}))dx\\
 &-2\int_{0}^{L}f_{1}e^{-\mu x}(E\mathbf{u})_{1}(EB)_{1}(\mathbf{u},x,t)+f_{2}e^{\mu x}(E\mathbf{u})_{2}(EB)_{2}(\mathbf{u},x,t)dx\\
  &-\mu\int_{0}^{L}\lambda_{1}f_{1}e^{-\mu x}(E\mathbf{u})_{1}^{2}+\lambda_{2}f_{2}e^{\mu x}(E\mathbf{u})_{2}^{2}dx
 +2qZ(t)\dot Z(t)\\
 &+C\left(\lVert \mathbf{u}\rVert_{\infty}+\lVert \partial_{x} \mathbf{u}\rVert_{\infty}\right)\int_{0}^{L}(E\mathbf{u})_{1}^{2}+(E\mathbf{u})_{2}^{2} dx\\
 &+C\left(\lVert \mathbf{u}\rVert_{\infty}+\lVert \partial_{x} \mathbf{u}\rVert_{\infty}\right)^{2}\int_{0}^{L}|(E\mathbf{u})_{1}|+|(E\mathbf{u})_{2}|dx,\\
\end{split}
\end{equation}
where $C$ is a constant that may change between lines but 
only depends on $\nu$, an upper bound of $\delta$ (for instance $\delta_{0}$), $\mu$, $H_{\max}$ and $\alpha$. Note that $C$ is continuous in $\mu\in[0,\infty)$, thus it can be made independent of $\mu$ by imposing an upper bound on $\mu$, for instance $\mu\in(0,1]$.
Finally, from the second equation of \eqref{definv}, and the fact that $E$ is $C^{1}$ in $\mathbf{u}$,
there exists a continuous function \textcolor{black}{$\mathcal{E}_{1}$} defined on $\mathcal{B}_{\nu_{2}}\times[0,L]\times[0,T]$ such that, for any vector $\mathbf{v}\in\mathbb{R}^{2}$
\begin{equation}
 E(\mathbf{u}(t,x),x,t)\mathbf{v}-\mathbf{v}=(\mathbf{u}(t,x).\mathcal{E}_{1}(\mathbf{u}(t,x),x,t))\mathbf{v},\text{ }\forall\text{ }(t,x)\in[0,T]\times[0,L].
 \label{recoverv}
\end{equation} 
As $E(\mathbf{u}(t,x),x,t)$ is a $C^{\infty}$ function of the coefficients of $A$, $\mathcal{E}_{1}$ is bounded on $\mathcal{B}_{\nu_{2}}\times[0,L]\times[0,T]$ by a bound that only depends on $\nu_{2}$, $H_{\max}$ and $\alpha$.
Thus there exists a constant $\bar C$ depending only on $\nu_{2}$, $H_{\max}$ and $\alpha$ such that
\textcolor{black}{\begin{equation}
\frac{1}{\bar C}\lVert\mathbf{v}\rVert_{\textcolor{black}{L^{2}((0,L);\mathbb{R}^{2})}}\leq \lVert E\mathbf{v}\rVert_{\textcolor{black}{L^{2}((0,L);\mathbb{R}^{2})}}\leq \bar C \lVert\mathbf{v}\rVert_{\textcolor{black}{L^{2}((0,L);\mathbb{R}^{2})}}.
\label{equivinvu}
\end{equation} }
Thus, using this together with the fact that $D_{1}$ and $D_{2}$ are $C^{1}$ with $\mathbf{u}$, \eqref{Va}, and Young's inequality and then Cauchy-Schwarz inequality on the last integral term,
\begin{equation}
\begin{split}
\dot V_{a}\leq& -\left[f_{1}e^{-\mu x}\lambda_{1}(E\mathbf{u})_{1}^{2}-\lambda_{2}f_{2}e^{\mu x}(E\mathbf{u})_{2}^{2}\right]_{0}^{L}\\
&-\int_{0}^{L}(E\mathbf{u})_{1}^{2}e^{-\mu x}(-\partial_{x}(\lambda_{1}f_{1})-\partial_{t}(f_{1}))+(E\mathbf{u})_{2}^{2}e^{\mu x}(\partial_{x}(\lambda_{2}f_{2})-\partial_{t}(f_{2}))dx\\
&-2\int_{0}^{L}f_{1}e^{-\mu x}(E\mathbf{u})_{1}(EB)_{1}(\mathbf{u},x,t)+f_{2}e^{\mu x}(E\mathbf{u})_{2}(EB)_{2}(\mathbf{u},x,t)dx\\
&-\mu\min\limits_{x\in[0,L]}(\lambda_{1},\lambda_{2})V_{a}\textcolor{black}{+\mu\min\limits_{x\in[0,L]}(\lambda_{1},\lambda_{2})qZ^{2}(t)}
+2qZ(t)\dot Z(t)\\
&+C\left(\lVert \mathbf{u}\rVert_{\infty}+\lVert \partial_{x} \mathbf{u}\rVert_{\infty}\right)\lVert\mathbf{u} \rVert_{L^{2}(0,L)}^{2}\\
&+C\left(\lVert \mathbf{u}\rVert_{\infty}+\lVert \partial_{x} \mathbf{u}\rVert_{\infty}\right)^{3}+C\lVert \mathbf{u}\rVert_{\infty}(|\mathbf{u}(t,0)^{2}|+|\mathbf{u}(t,L)^{2}|).
\end{split}
\end{equation} 
Now, as $E$ and $B$ are $C^{2}$ with $\mathbf{u}$ and continuous with $x$ and $t$, and as $B(\mathbf{0},x,t)=0$, there exists a continuous function \textcolor{black}{$\mathcal{E}_{2}\in C^{0}(\mathcal{B}_{\nu_{2}}\times[0,T]\times[0,L];\mathbb{R}^{2\times 2\times 2})$} such that,
\begin{equation}
 (EB)(\mathbf{u}(t,x),x,t)=\partial_{\mathbf{u}}(EB)(\mathbf{0},x,t).\mathbf{u}(t,x)+(\mathcal{E}_{2}(\mathbf{u},x,t).\mathbf{u}(t,x))\mathbf{u}(t,x),\text{ }\forall\text{ }t\in[0,T]\times[0,L].
 \label{Bprop}
\end{equation} 
Note that from \eqref{defB}, $\mathcal{E}_{2}$ is bounded on $\mathcal{B}_{\nu_{2}}\times[0,L]\times[0,T]$ by a constant that only depends on $\nu_{2}$, $\delta$, $H_{\max}$ and $\alpha$. From \eqref{defB} and \eqref{definv}
$\partial_{\mathbf{u}}(EB)(\mathbf{0},x,t)=\partial_{\mathbf{u}}B(\mathbf{0},x,t)$. Besides, from \eqref{definv}, $E$ is invertible and $C^{1}$, thus an inequality similar to 
\eqref{equivinvu}
 holds for $E^{-1}$, and $\mathbf{u}=E^{-1}(E\mathbf{u})$.
Therefore, using \eqref{Bprop} together with \eqref{recoverv}, \textcolor{black}{the fact that $\mathcal{E}_{1}$ and $\mathcal{E}_{2}$ are bounded,} and the expression of $\partial_{\mathbf{u}}B(\mathbf{0},x,t)$ given in \textcolor{black}{\eqref{l1l2}}--\eqref{gamma}, one has
\begin{equation}
\begin{split}
\dot V_{a}\leq& -\left[f_{1}e^{-\mu x}\lambda_{1}u_{1}^{2}-\lambda_{2}f_{2}e^{\mu x}u_{2}^{2}\right]_{0}^{L}\\
&-\int_{0}^{L}(E\mathbf{u})_{1}^{2}e^{-\mu x}(-\partial_{x}(\lambda_{1}f_{1})-\partial_{t}(f_{1}))+(E\mathbf{u})^{2}e^{\mu x}(\partial_{x}(\lambda_{2}f_{2})-\partial_{t}(f_{2}))dx\\
&-2\int_{0}^{L}f_{1}e^{-\mu x}\gamma_{1}(E\mathbf{u})_{1}^{2}+f_{2}e^{\mu x}\delta_{2}(E\mathbf{u})_{2}^{2}+\left(\gamma_{2}f_{1}e^{-\mu x}+\delta_{1}f_{2}e^{\mu x}\right)(E\mathbf{u})_{1}(E\mathbf{u})_{2}dx\\
&-\mu\min\limits_{x\in[0,L]}(\lambda_{1},\lambda_{2})V_{a}\textcolor{black}{+\mu\min\limits_{x\in[0,L]}(\lambda_{1},\lambda_{2})qZ^{2}(t)}
+2qZ(t)\dot Z(t)\\
&+C\left(\lVert \mathbf{u}\rVert_{\infty}+\lVert \partial_{x} \mathbf{u}\rVert_{\infty}\right)\lVert\mathbf{u} \rVert_{L^{2}(0,L)}^{2}\\
&+C\left(\lVert \mathbf{u}\rVert_{\infty}+\lVert \partial_{x} \mathbf{u}\rVert_{\infty}\right)^{3}+C\lVert \mathbf{u}\rVert_{\infty}(|\mathbf{u}(t,0)^{2}|+|\mathbf{u}(t,L)^{2}|).
\end{split}
\end{equation} 
As $\mathcal{D}_{1}$ and $\mathcal{D}_{2}$ are of class $C^{2}$, 
denoting for simplicity $k_{2}:=\partial_{1} \mathcal{D}_{1}(0,t)$\textcolor{black}{, $k_{1}:=\partial_{1}\mathcal{D}_{2}(0,0,t)$ 
and $k_{3}:=-\partial_{2}\mathcal{D}_{2}(0,0,t)$, and using \eqref{bound}}

\begin{equation}
\begin{split}
 \dot V_{a}&\leq
-\mu\min\limits_{x\in[0,L]}(\lambda_{1},\lambda_{2})V_{a}+\left[f_{1}\lambda_{1}k_{2}^{2}-\lambda_{2}f_{2}\right]u_{2}^{2}(t,0)\\
 &-I_{1}(u_{1}(t,L),Z(t))
 -\int_{0}^{L}I_{2}((E\mathbf{u})_{1},(E\mathbf{u})_{2})dx\\
 &+C\left(\lVert \mathbf{u}\rVert_{\infty}+\lVert \partial_{x} \mathbf{u}\rVert_{\infty}\right)\left(\lVert\mathbf{u} \rVert_{L^{2}(0,L)}^{2}+\left(\lVert \mathbf{u}\rVert_{\infty}+\lVert \partial_{x} \mathbf{u}\rVert_{\infty}\right)^{2}+(|\mathbf{u}(t,0)^{2}|+|\mathbf{u}(t,L)^{2}|)\right),
 \end{split}
 \label{diffVa}
\end{equation}
where $I_{1}$ and $I_{2}$ denote the following quadratic forms
\begin{equation}
\begin{split}
 I_{1}(x,y)=&\left(\lambda_{1}f_{1}(L)e^{-\mu L}-\lambda_{2}f_{2}(L)e^{\mu L}k_{1}^{2}\right)x^{2}\\
 &+\textcolor{black}{\left(q\sqrt{\frac{H_{1}}{g}}k_{3}-\lambda_{2}f_{2}(L)e^{\mu L}k_{3}^{2}-\mu\min\limits_{x\in[0,L]}(\lambda_{1},\lambda_{2})q\right)}y^{2}\\
&+(2\lambda_{2}f_{2}(L)e^{\mu L}k_{3}k_{1}-\textcolor{black}{q\sqrt{\frac{H_{1}}{g}}}\left(k_{1}-1\right))xy,\\
 I_{2}(x,y)=&\left((-\lambda_{1}f_{1})_{x}+2f_{1}\gamma_{1}(t,x)-\partial_{t}f_{1}\right)e^{-\mu x}x^{2}+\left((\lambda_{2}f_{2})_{x}+2f_{2}\delta_{2}(t,x)-\partial_{t}f_{2}\right)e^{\mu x}y^{2}\\
 &+2\left(\gamma_{2}f_{1}e^{-\mu x}+\delta_{1}f_{2}e^{\mu x}\right)xy.
 \end{split}
 \label{I1I2}
\end{equation} 
We can perform similarly with $V_{b}$ and $V_{c}$, to do this observe that $\partial_{t}\mathbf{u}$ and $\partial_{tt}^{2}\mathbf{u}$ are respectively solutions of
\begin{gather}
\begin{split}
&\partial_{t}(\partial_{t}\mathbf{u})+A(\mathbf{u},x,t)\partial_{x}(\partial_{t}\mathbf{u})+(\partial_{\mathbf{u}}B(\mathbf{u},x,t))(\partial_{t}\mathbf{u})
+(\partial_{t}A(\mathbf{u},x,t)+\partial_{\mathbf{u}}A(\mathbf{u},x,t).\partial_{t}\mathbf{u})\partial_{x}\mathbf{u}\\
&+\partial_{t}B\left(\mathbf{u},x,t\right)=0
\end{split}
\label{V2}\\
\begin{split}
&\partial_{t}(\partial_{tt}^{2}\mathbf{u})+A(\mathbf{u},x,t)\partial_{x}(\partial_{tt}^{2}\mathbf{u})
+(\partial_{\mathbf{u}}A(\mathbf{u},x).\partial_{tt}^{2}\mathbf{u})\partial_{x}\mathbf{u}
+(\partial_{\mathbf{u}}B(\mathbf{u},x))(\partial_{tt}^{2}\mathbf{u})\textcolor{black}{,}\\
&+2\partial_{\mathbf{u}}(\partial_{t}A(\mathbf{u},x,t)).\partial_{t}\mathbf{u})\partial_{x}\mathbf{u}+(\partial_{tt}^{2}A(\mathbf{u},x,t)+2\partial_{\mathbf{u}}A(\mathbf{u},x).\partial_{t}\mathbf{u}\partial_{x}(\partial_{t}\mathbf{u})+\partial_{t}A(\mathbf{u},x,t)\partial_{x}(\partial_{t}\mathbf{u})\\
&+((\partial^{2}_{\mathbf{u}}A(\mathbf{u},x).\partial_{t}\mathbf{u}).\partial_{t}\mathbf{u})\partial_{x}\mathbf{u}+\partial_{tt}^{2}B(\mathbf{u},x)+\partial_{\mathbf{u}}(\partial_{t}B(\mathbf{u},x)).\partial_{t}\mathbf{u}
+(\partial^{2}_{\mathbf{u}}B(\mathbf{u},x).\partial_{t}\mathbf{u})(\partial_{t}\mathbf{u})=0\textcolor{black}{,}
\label{V3}
\end{split}
\end{gather} 
which are very similar to \eqref{sys1}, as they only differ by quadratic perturbations or terms involving a time derivative of $(H_{1},V_{1})$.
We get then
\begin{equation}
\begin{split}
&\dot V=\dot V_{a}+\dot V_{b}+\dot V_{c}\leq-\mu\min\limits_{x\in[0,L]}(\lambda_{1},\lambda_{2})V\\
&+\left[f_{1}\lambda_{1}k_{2}^{2}-\lambda_{2}f_{2}\right]\left(u_{2}^{2}(t,0)+(\partial_{t}u_{2}(t,0))^{2}+(\partial_{tt}^{2}u_{2}(t,0))^{2}\right)\\
&-I_{1}(u_{1}(t,L),Z)-I_{1}(\partial_{t}u_{1}(t,L),\dot Z)-I_{1}(\partial_{tt}^{2}u_{1}(t,L),\ddot Z)\\
&-\int_{0}^{L}I_{2}((E\mathbf{u})_{1},(E\mathbf{u})_{2})+I_{2}((E\partial_{t}\mathbf{u})_{1},(E\partial_{t}\mathbf{u})_{2})+I_{2}((E\partial_{tt}^{2}\mathbf{u})_{1},(E\partial_{tt}^{2}\mathbf{u})_{2})dx\\
&+C\left(\lVert \mathbf{u}\rVert_{\infty}+\lVert \partial_{x} \mathbf{u}\rVert_{\infty}\right)\left(\lVert\mathbf{u} \rVert_{L^{2}(0,L)}^{2}+\lVert\partial_{t}\mathbf{u} \rVert_{L^{2}(0,L)}^{2}
+\lVert\partial_{tt}^{2}\mathbf{u} \rVert_{L^{2}(0,L)}^{2}+\left(\lVert \mathbf{u}\rVert_{\infty}+\lVert \partial_{x} \mathbf{u}\rVert_{\infty}\right)^{\textcolor{black}{2}}\right.\\
&\left.+|u_{2}(t,0)^{2}|+(|u_{1}(t,L)|+|Z|)^{2}+|\partial_{t}u_{2}(t,0)^{2}|+(|\partial_{t}u_{1}(t,L)|+|\dot Z|)^{2}+|\partial_{tt}^{2}u_{2}(t,0)^{2}|\right.\\
&+\left.(|\partial_{tt}^{2}u_{1}(t,L)|+|\ddot Z|)^{2}\vphantom{\left(\lVert\rVert\right)^{2}}\right)\\
&+C\delta\left(|u_{2}(t,0)|^{2}+(|u_{1}(t,L)|+|Z|)^{2}+|\partial_{t}u_{2}(t,0)|^{2}+(|\partial_{t}u_{1}(t,L)|+|\dot Z|)^{2}\right)+C\delta V.
\end{split}
\label{dVtot}
\end{equation} 

The two last terms come from the successive differentiations of the boundary conditions \eqref{bound}, together with \eqref{boundHV}, or the terms in \eqref{V2}--\eqref{V3} involving a time derivative of $A$ or $B$.
One can see that three identical quadratic form appears in the integral in $((E\partial_{t}^{i}\mathbf{u})_{1},(E\partial_{t}^{i}\mathbf{u})_{2}),$ $i=0,1,2$, as well as three identical quadratic form at the boundaries in $(\partial_{t}^{i}u_{1}(t,L),\partial_{t}^{i}Z),$ $i=0,1,2$, and three identical terms proportional respectively to $(\partial_{t}^{i}u_{2}(t,0)),$ $i=0,1,2$.
Thus a sufficient condition to have $V$ strictly decreasing would be that the square terms and the forms that appear at the boundaries are negative-definite and the quadratic form in the integral is negative, i.e. the three following conditions:\\
\begin{enumerate}
 \item  \textbf{Condition at $0$}
 \begin{equation}
  \frac{\lambda_{2}f_{2}(0)}{\lambda_{1}f_{1}(0)}>k_{2}^{2}\textcolor{black}{.}
  \label{c1}
 \end{equation} 
\item \textbf{Condition at $L$}
 \begin{subequations}
\begin{align} \frac{\lambda_{1}f_{1}(L)}{\lambda_{2}f_{2}(L)}&>k_{1}^{2},
 \label{c2a}\\
\left(\lambda_{1}f_{1}(L)-\lambda_{2}f_{2}(L)k_{1}^{2}\right)\left(q\sqrt{\frac{H_{1}}{g}}-\lambda_{2}f_{2}(L)k_{3}\right)k_{3}&-\left(\lambda_{2}f_{2}(L)k_{3}k_{1}-\frac{1}{2}q\sqrt{\frac{H_{1}}{g}}\left(k_{1}-1\right)\right)^{2}>0\textcolor{black}{.}      \label{c2b}
      \end{align}
\label{c2}
\end{subequations}
      \item \textbf{Condition from the integral}
 \begin{subequations}
\begin{align}
      &\left((-\lambda_{1}f_{1})_{x}+2f_{1}\gamma_{1}(t,x)-\partial_{t}f_{1}\right)>0,
      \label{c31}\\
      \begin{split}
     &\left((-\lambda_{1}f_{1})_{x}+2f_{1}\gamma_{1}(t,x)-\partial_{t}f_{1}\right)\left((\lambda_{2}f_{2})_{x}+2f_{2}\delta_{2}(t,x)-\partial_{t}f_{2}\right)
          \\ & -\left(\gamma_{2}f_{1}+\delta_{1}f_{2}\right)^{2}>0,\text{  }\forall\text{  }(t,x)\in [0,T]\times(0,L).
           \label{c3b}
           \end{split}
           \end{align}
           \label{c3}
           \end{subequations} 
\end{enumerate}
Let assume for the moment that \eqref{c1}--\eqref{c3} are satisfied
for any $\delta\in(0,\delta_{3})$ where $\delta_{3}$ is a positive constant.
Then, as the inequalities \eqref{c1}--\eqref{c3} are strict, by continuity there exist $\mu>0$ such that the square terms and the quadratic forms \textcolor{black}{$I_{1}$} at the boundaries and the quadratic forms \textcolor{black}{$I_{2}$} in the integral are positive definite. And there exists $\nu_{3}\in(0,\nu_{2})$ 
and $\delta_{4}\in(0,\delta_{3})$
such that, for any $\nu\in(0,\nu_{3})$,
and any $\delta\in(0,\delta_{4})$,
\begin{equation}
 \dot V\leq -\mu\min\limits_{[0,L]}(\lambda_{1},\lambda_{2})V+C\delta V +C\left(\left(\lVert\mathbf{u}\rVert_{\infty}+\lVert\partial_{x}\mathbf{u}\rVert_{\infty}\right)^{2}\right)\textcolor{black}{,}
\end{equation} 
where $C$ is a positive constant depending only on the system. Note that here, the cubic boundary terms that appeared in \eqref{dVtot} have been compensated by the strictly negative quadratic boundary terms, taking $\nu$ sufficiently small and using \eqref{estimate}.
Choosing $\delta_{5}\in(0,\delta_{4})$ such that $\delta_{5}<\mu\min_{[0,L]}(\lambda_{1},\lambda_{2})/4C$, for any $\delta\in(0,\delta_{5})$ one has
\begin{equation}
  \dot V\leq -\frac{3}{4}\mu\min\limits_{[0,L]}(\lambda_{1},\lambda_{2})V+C\left(\left(\lVert\mathbf{u}\rVert_{\infty}+\lVert\partial_{x}\mathbf{u}\rVert_{\infty}\right)^{2}\right)\textcolor{black}{.}
\end{equation} 
Now, if we assume in addition that \eqref{cV1} hold, using \eqref{estimate}, and Sobolev inequality, there exists $\nu_{4}\in(0,\nu_{3}]$ such that, for any $\nu\in(0,\nu_{4})$,
\begin{equation}
C\left(\left(\lVert\mathbf{u}\rVert_{\infty}+\lVert\partial_{x}\mathbf{u}\rVert_{\infty}\right)^{2}\right)\leq \frac{\mu}{4}\min\limits_{[0,L]}(\lambda_{1},\lambda_{2})V,
\end{equation} 
thus, setting $\gamma=\mu\min_{[0,L]}(\lambda_{1},\lambda_{2})$,
\begin{equation}
\dot V\leq -\frac{\gamma}{2} V.
\end{equation} 
which shows the exponential decay of $V$ and
ends the proof of Theorem \ref{th1}.\\

In other words, all that remains to do is to find $f_{1}$, $f_{2}$ and $q$ such that \eqref{c1}--\eqref{c3} are satisfied and such that $V$ satisfies \eqref{cV1}.
In order to find such functions, we use a lemma below.

Let us first introduce the following function $\phi$ defined by
\begin{equation}
\begin{split}
&\phi_{1}(t,x)=\exp\left(\int_{0}^{x}\frac{\gamma_{1}}{\lambda_{1}}dx\right),\\
&\phi_{2}(t,x)=\exp\left(-\int_{0}^{x}\frac{\delta_{2}}{\lambda_{2}}dx\right),\\
&\phi(t,x)=\frac{\phi_{1}(t,x)}{\phi_{2}(t,x)},
\end{split}
\label{phi}
\end{equation}
where $\lambda_{1}$ and $\lambda_{2}$ are defined by
\begin{equation}
\lambda_{1}(t,x):=\Lambda_{1}(\mathbf{0},x,t)>0,\text{  }\lambda_{2}(t,x):=-\Lambda_{2}(\mathbf{0},x,t)>0.
\label{lambda}
\end{equation} 
Note that $\lambda_{2}$ is not the second eigenvalue (when $\mathbf{u} = 0$) but its opposite. We use this notation so that $\lambda_{2}>0$.

We can state the following lemma
\begin{lem}
There exists $\delta_{0}>0$ such that if $\lVert\partial_{t}H_{1}\rVert_{L^{\infty}((0,+\infty)\times(0,L)}\leq \delta_{0}$,
the function $\chi=\textcolor{black}{\lambda_{2}\phi}/\lambda_{1}$ is solution on $[0,L]$ to the following equation
\begin{equation}
\partial_{x}\chi
=\left|\frac{\phi\gamma_{2}}{\lambda_{1}}+\frac{\phi^{-1}\delta_{1}}{\lambda_{2}}\chi^{2}+\sqrt{\frac{g}{H_{1}}}\partial_{t} H_{1}\right|, \forall\text{  }x\in[0,L],\text{  }t\in[0,+\infty)\textcolor{black}{,}
\label{eqg}
\end{equation}
and for any $x\in[0,L]$ and any $t\in[0,+\infty)$,
\begin{equation}
\left(\frac{\phi\gamma_{2}}{\lambda_{1}}+\frac{\phi^{-1}\delta_{1}}{\lambda_{2}}\chi^{2}+\sqrt{\frac{g}{H_{1}}}\partial_{t} H_{1}\right)>0.
\label{positivity}
\end{equation}
\label{lem2}
\end{lem}
The proof is given in the Appendix.\\

To understand the link between Lemma \ref{lem2} and the three conditions \eqref{c1}--\eqref{c3}, observe that the condition \eqref{c3} give rise to a differential \textcolor{black}{inequation}, which, as it will appear later on, 
is linked to the differential equation solved by Lemma \ref{lem2}. Then \eqref{c1} and \eqref{c2} can be seen as boundary conditions/values of the solution of this differential \textcolor{black}{inequation}.\\

Let now assume that $\delta <\delta_{0}$, where $\delta_{0}$ is given by Lemma \ref{lem2}. From Lemma \ref{lem2}, we know that there exists a solution on $[0,L]$ to equation \eqref{eqg}, namely $\lambda_{2}\phi/\lambda_{1}$.
Therefore, as $[0,L]$ is a compact set, there exists $\varepsilon_{1}$ such that 
for any $\varepsilon\in[0,\varepsilon_{1})$ there exists a solution $\chi_{\varepsilon}(t,x)$ to the following system
\begin{equation}
\begin{split}
&\partial_{x}\chi_{\varepsilon}(t,x)=\left(\frac{\phi\gamma_{2}}{\lambda_{1}}+\frac{\delta_{1}}{\phi\lambda_{2}}(\chi_{\varepsilon})^{2}+\sqrt{\frac{g}{H_{1}}}\partial_{t} H_{1}\right)+\varepsilon\textcolor{black}{,}\\
&\chi_{\varepsilon}(0)=\frac{\lambda_{2}(t,0)}{\lambda_{1}(t,0)}+\varepsilon\textcolor{black}{,}
\end{split}
\label{eqf}
\end{equation}
and moreover $(t,x,\varepsilon)\mapsto \chi_{\varepsilon}(t,x)$ is of class $C^{0}$ and $\partial_{x}\chi_{\varepsilon}(t,x)$ as well. 
This is a classical result on ODE due to Peano (see e.g. \cite{Hartman}[Chap. 5, Th 3.1]). From \eqref{eqf},
$\partial_{t}\chi_{\varepsilon}$ satisfies the following equation
\begin{equation}
\partial_{x}\partial_{t}\chi_{\varepsilon}=2\frac{\delta_{1}}{\phi\lambda_{2}}\chi_{\varepsilon}\partial_{t}\chi_{\varepsilon}+\left(\frac{\phi\gamma_{2}}{\lambda_{1}}\right)_{t}+\left(\frac{\delta_{1}}{\phi\lambda_{2}}\right)_{t}\chi^{2}_{\varepsilon}
+\sqrt{\frac{g}{H_{1}}}\partial_{tt}^{2} H_{1}-\frac{1}{2}\sqrt{\frac{g}{H_{1}^{3}}}(\partial_{t}H_{1})^{2}\textcolor{black}{.}
\label{eqdtf}
\end{equation} 
We used here that, from Proposition \ref{propISS} and Remark \ref{rmkHpISS}, $(H_{1},V_{1})\in C^{0}([0,+\infty); H^{3}(0,L))$, and from \eqref{target}, $\partial_{t}\partial_{x}H_{1}=-\partial_{x}^{2}(HV)$ and 
$\partial_{t}\partial_{x}V_{1}=\partial_{x}\left(-V_{1}\partial_{x}V_{1}-g\partial_{x}H_{1}-(kV_{1}^{2}/H_{1}-gC)\right)$. Thus $\partial_{tt}^{2}H_{1}$ belongs to $C^{0}([0,T]; H^{1}(0,L))$, and $(\gamma_{1}, \gamma_{2},\delta_{1},\delta_{2})$ belong to $C^{1}([0,T]; H^{1}(0,L))$.
Using \eqref{eqdtf}, we have
\begin{equation}
\begin{split}
\partial_{t}\chi_{\varepsilon}(t,x)=&\partial_{t}\chi_{\varepsilon}(t,0)\exp\left(\int_{0}^{x}2\frac{\delta_{1}}{\phi\lambda_{2}}\chi_{\varepsilon}(t,y) dy\right)\\
&+\int_{0}^{x}\exp\left(\int_{y}^{x}2\frac{\delta_{1}}{\phi\lambda_{2}}\chi_{\varepsilon}(t,\omega) d\omega\right)\left(\left(\frac{\phi\gamma_{2}}{\lambda_{1}}\right)_{t}+\left(\frac{\delta_{1}}{\phi\lambda_{2}}\right)_{t}\chi^{2}_{\varepsilon}+\sqrt{\frac{g}{H_{1}}}\partial_{tt}^{2} H_{1}-\frac{1}{2}\sqrt{\frac{g}{H_{1}^{3}}}(\partial_{t}H_{1})^{2} \right)dy.
\end{split}
\label{difffeps}
\end{equation} 
Instead of seeing the function $\chi_{\varepsilon}$ as a solution of an ODE with a parameter $t$, one can see it as a solution of an ODE with parameters 
$\lambda_{1}$, $\lambda_{2}$, $\gamma_{2}$, $\delta_{1}$, $\partial_{t}H_{1}$ and $\varepsilon$ that we denote $g_{\varepsilon}(x,\lambda_{1},\lambda_{2},\gamma_{1},\delta_{1},\partial_{t}H_{1})$. From \cite{Hartman}[Theorem 2.1] $g_{\varepsilon}$ 
is continuous with these parameters and with $\varepsilon$. But from \eqref{Hinfty}, \eqref{fluvial}, and \eqref{boundHV}, all these parameters are bounded and therefore belong to a compact set when $t\in[0,+\infty)$.
Thus, 
\begin{equation}
\varepsilon\rightarrow g_{\varepsilon}(x,\lambda_{1}(t),\lambda_{2}(t),\gamma_{1}(t),\delta_{1}(t),\partial_{t}H_{1}(t))=\chi_{\varepsilon}(t,x)
\end{equation} 
is uniformly continuous in $\varepsilon$ for $(t,x)\in[0,\infty)\times[0,L]$. This, together with \eqref{difffeps} implies that there exists $C_{0}$ depending only on $L$, $H_{\max}$, $\alpha$, $\varepsilon$ and continuous with $\varepsilon\in[0,\varepsilon_{1})$ such that
\begin{equation}
\begin{split}
&\left|\int_{0}^{x}\exp\left(\int_{y}^{x}2\frac{\delta_{1}}{\phi\lambda_{2}}\chi_{\varepsilon}(t,\omega) d\omega\right)\partial_{t}(\partial_{y}(H_{1}V_{1}))dy\right|\\
&\leq C_{0} \max\left(\lVert\partial_{t}H_{1}\rVert_{C^{1}([0,+\infty);C^{0}([0,L]))},\lVert\partial_{t}V_{1}\rVert_{C^{1}([0,+\infty);C^{0}([0,L]))}\right)\textcolor{black}{,}
\end{split}
\end{equation} 

Similarly
there exists a constant $C_{1}>0$ depending only on $L$, $H_{\max}$ and $\alpha$ such that
\begin{equation}
\rVert \partial_{t}\phi_{1}\lVert_{L^{\infty}((0,+\infty)\times(0,L)}\leq C_{1} \max\left(\lVert\partial_{t}H_{1}\rVert_{C^{1}([0,+\infty);C^{0}([0,L]))},\lVert\partial_{t}V_{1}\rVert_{C^{1}([0,+\infty);C^{0}([0,L]))}\right)\textcolor{black}{,}
\label{dphi}
\end{equation} 
and similarly for $\phi_{2}$. This, together with the definition of $\lambda_{1}$ and $\lambda_{2}$ given by \eqref{lambda}, \eqref{difffeps}, and
using the continuity of $\varepsilon\rightarrow \chi_{\varepsilon}$ on $[0,\varepsilon_{1})$ (recall that this continuity is uniform with respect to $(t,x)\in[0,+\infty)\times[0,L]$),
we get that there exists $C>0$ depending only on $H_{\max}$, $\alpha$ and $\varepsilon$ and continuous with $\varepsilon$ on $[0,\varepsilon_{1})$ such that
\begin{equation}
|\partial_{t}\chi_{\varepsilon}(t,x)|\leq (|\partial_{t}\chi_{\varepsilon}(t,0)|
\textcolor{black}{+ \max\left(\lVert\partial_{t}H_{1}\rVert_{C^{1}([0,+\infty);C^{0}([0,L]))},\lVert\partial_{t}V_{1}\rVert_{C^{1}([0,+\infty);C^{0}([0,L]))}\right))C(\varepsilon)}.
\end{equation} 
But, from \eqref{eqf} $\partial_{t}\chi_{\varepsilon}(t,0)=\left(\lambda_{2}/\lambda_{1}\right)_{t}$, thus 
using \eqref{boundHV}
we obtain
\begin{equation}
|\partial_{t}\chi_{\varepsilon}(t,x)|\leq \delta C_{2}(\varepsilon)\textcolor{black}{,}
\label{borne_dtfeps}
\end{equation} 
where $C_{2}$ is again a constant that only depends on $\varepsilon$, $\alpha$ and $H_{\max}$ and is continuous with $\varepsilon$ on $[0,\varepsilon_{1})$. 
We can now restrict ourselves to $\varepsilon\in[0,\varepsilon_{1}/2]$ and then $C_{2}$ can be chosen independent of $\varepsilon$ by simply taking its maximum on $[0,\varepsilon_{1}/2]$.
Recall that from Lemma \ref{lem2} we have, $\chi_{0}=\phi\lambda_{2}/\lambda_{1}$, and
\begin{equation}
\left(\frac{\phi\gamma_{2}}{\lambda_{1}}+\frac{\delta_{1}}{\phi\lambda_{2}}\chi_{0}^{2}+\sqrt{\frac{g}{H_{1}}}\partial_{t} H_{1}\right)>0.
\end{equation} 
Recall that we still have not chosen the bound $\delta\in(0,\textcolor{black}{\delta_{0}})$ on $\lVert\partial_{t}H_{1}\rVert_{C^{1}([0,\infty);C^{0}([0,L]))}$ and $\lVert\partial_{t}V_{1}\rVert_{C^{1}([0,\infty);C^{0}([0,L]))}$ given in \eqref{boundHV}.
From the assumptions on $k_{p}$ and $k_{I}$, i.e. \eqref{cond}, and \eqref{k1k3}, \textcolor{black}{and recalling that $k_{1}=\partial_{1}\mathcal{D}_{2}(0,0,t)$ 
and $k_{3}=-\partial_{2}\mathcal{D}_{2}(0,0,t)$}, one has
\begin{equation}
k_{1}^{2}<\left(\frac{\lambda_{1}(L)}{\lambda_{2}(L)}\right)^{2},\text{    }\text{    }\text{    }k_{3}>0.
\label{condk1k3}
\end{equation}
Thus, using \eqref{lambda},
\begin{equation}
\eta_{1}:=\min\left(\left(\frac{1}{|k_{1}|}-\frac{\lambda_{2}(L)}{\lambda_{1}(L)}\right),1-\frac{\lambda_{2}(L)}{\lambda_{1}(L)}\right)>0.
\label{defeta1}
\end{equation} 
As $\varepsilon\rightarrow \chi_{\varepsilon}(t,x)$ is uniformly continuous with $\varepsilon$ for $(t,x)\in[0,\infty)\times[0,L]$,
there exists
$\varepsilon_{2}\in(0,\varepsilon_{1}/2)$ such that
for any $(t,x)\in[0,\infty)\times[0,L]$
\begin{equation}
|\chi_{\varepsilon_{2}}(t,x)-\chi_{0}(t,x)|\leq \phi(t,L)\eta_{1},
\label{fL}
\end{equation}
and
\begin{equation}
\left(\frac{\phi\gamma_{2}}{\lambda_{1}}+\frac{\delta_{1}}{\phi\lambda_{2}}\chi_{\varepsilon_{2}}^{2}+\sqrt{\frac{g}{H_{1}}}\partial_{t} H_{1}\right)>0.
\label{fpositive}
\end{equation}
Note that $\varepsilon_{2}$ depends \textcolor{black}{$a$ $priori$} on $\delta$ from \eqref{fpositive}. However, from Lemma \ref{lem2} 
we can in fact choose $\varepsilon_{2}$ independent of $\delta$ and depending only on an upper bound of $\delta$ (for instance $\delta_{0}$ given by Lemma \ref{lem2}). This is important as, in the following,
we will choose
a $\delta$ that may depends on $\varepsilon$.\\

We select $f_{1}$ and $f_{2}$ in the following way:
\begin{equation}
\begin{split}
&f_{1}(t,x)=\frac{\phi^{2}_{1}}{\lambda_{1}\chi_{\varepsilon_{2}}(t,x)}>0\textcolor{black}{,}\\
&f_{2}(t,x)=\phi_{2}^{2}\frac{\chi_{\varepsilon_{2}}(t,x)}{\lambda_{2}}>0,
\label{deff1f2}
\end{split}
\end{equation} 
and we can now check that the condition \eqref{c3} is verified for $\delta$ small enough as
\begin{equation}
(-\lambda_{1}f_{1})_{x}=-2\frac{(\phi_{1})_{x}\lambda_{1}f_{1}}{\phi_{1}}+\phi_{1}^{2}\frac{\partial_{x}\chi_{\varepsilon_{2}}(t,x)}{\chi_{\varepsilon_{2}}^{2}(t,x)}.
\end{equation} 
Thus from \eqref{phi}
\begin{equation}
-(\lambda_{1}f_{1})_{x}+2\gamma_{1}f_{1}
=\phi_{1}^{2}\frac{\partial_{x}\chi_{\varepsilon_{2}}}{\chi^{2}_{\varepsilon_{2}}}
\label{f1}
\end{equation} 
and similarly
\begin{equation}
\begin{split}
(\lambda_{2}f_{2})_{x}+2\delta_{2}f_{2}&=(\phi^{2}_{2}\chi_{\varepsilon_{2}}(t,x))_{x}-(\phi^{2}_{2})_{x}\chi_{\varepsilon_{2}}(t,x)\\
&=\phi_{2}^{2}\partial_{x}\chi_{\varepsilon_{2}}.
\end{split}
\label{f2}
\end{equation} 
Therefore, from \eqref{eqf}, \eqref{f1}, and \eqref{f2}, one has
\begin{equation}
\begin{split}
(-(\lambda_{1}f_{1})_{x}+2\gamma_{1}f_{1}-\partial_{t}f_{1})((\lambda_{2}f_{2})_{x}+2\delta_{2}f_{2}-\partial_{t}f_{2})
=&\left(\frac{\phi_{1}\phi_{2}}{\chi_{\varepsilon_{2}}}\right)^{2}\left(\left(\frac{\phi \gamma_{2}}{\lambda_{1}}+\frac{\delta_{1}}{\phi\lambda_{2}}\chi_{\varepsilon_{2}}^{2}+\sqrt{\frac{g}{H_{1}}}\partial_{t} H_{1}\right)
+\varepsilon_{2}\right)^{2}\\
&-\partial_{x}\chi_{\varepsilon_{2}}\left(\frac{\phi_{1}^{2}}{\chi^{2}_{\varepsilon_{2}}}\partial_{t}f_{2}
+\phi_{2}^{2}\partial_{t}f_{1}\right)
+(\partial_{t}f_{1})(\partial_{t}f_{2})\textcolor{black}{.}
\label{cond00}
\end{split}
\end{equation} 
But we have
\begin{equation}
\partial_{t}f_{1}=2\frac{(\partial_{t}\phi_{1})\phi_{1}}{\lambda_{1}\chi_{\varepsilon_{2}}}-(\frac{\partial_{t}\lambda_{1}}{\lambda_{1}^{2}\chi_{\varepsilon_{2}}}+\frac{\partial_{t}\chi_{\varepsilon_{2}}}{\lambda_{1}\chi^{2}_{\varepsilon_{2}}})\phi_{1}^2,
\end{equation} 
and besides, from \eqref{target} and \eqref{boundHV}, there exists $C_{3}>0$ depending only on $\alpha$ and $H_{\max}$, and an upper bound of $\delta$ (for instance $\textcolor{black}{\delta_{0}}$), such that
\begin{equation}
\max(\lVert H_{1x}\rVert_{L^{\infty}((0,+\infty)\times(0,L)},\lVert V_{1x}\rVert_{L^{\infty}((0,+\infty)\times(0,L)})\leq C_{3}.
\end{equation} 
Thus, using \eqref{gamma} and \eqref{lambda}, there exists $C_{4}>0$ depending only on $\alpha$ and $H_{\max}$, and $\delta_{0}$ (but not on $\delta$) such that
\begin{equation}
\max(\lVert\phi_{1}\rVert_{L^{\infty}((0,+\infty)\times(0,L)},\lVert\phi^{-1}\rVert_{L^{\infty}((0,+\infty)\times(0,L)})<C_{4}\textcolor{black}{,}
\label{boundphi}
\end{equation} 
and similarly for $\phi_{2}$. Observe now that, from $\chi_{0}=\lambda_{2}\phi/\lambda_{1}$ and \eqref{boundphi}, $|\chi_{0}|$ and $1/|\chi_{0}|$ can be bounded by a constant depending only on $\alpha$, $H_{\max}$, and 
$\delta_{0}$. Thus from 
\eqref{fL}
\begin{equation}
1/C_{5}\leq \lVert \chi_{\varepsilon_{2}}\rVert_{L^{\infty}((0,+\infty)\times(0,L)}\leq C_{5},
\label{boundf}
\end{equation}
where $C_{5}$ only depends on $\alpha$, $H_{\max}$ and $\delta_{0}$.
And therefore, from
\eqref{lambda}, \eqref{borne_dtfeps}, \eqref{dphi}, and \eqref{boundf} one has
\begin{equation}
|\partial_{t}f_{1}|\leq C_{6}\delta,
\end{equation} 
and similarly 
\begin{equation}
|\partial_{t}f_{2}|\leq C_{7}\delta,
\end{equation} 
where $C_{6}$ and $C_{7}$ are constants that only depend on $\alpha$, $H_{\max}$ (and $\delta_{0}$). 
We now select the bound on $\max\left(|\partial_{t}H_{1}|,|\partial_{t}V_{1}|\right)$: we select $\delta_{3}\in(0,\delta_{0})$ such that, for any $\delta\in[0,\delta_{3}]$ and any $(t,x)\in[0,\infty)\times[0,L]$,
\begin{equation}
C_{6}C_{5}^{2}C_{4}^{2}\delta<\varepsilon_{2}, 
\label{conddelta02}
\end{equation}
and
\begin{equation}
\begin{split}
&\varepsilon_{2}^{2}+2\varepsilon_{2}\inf\limits_{x\in[0,L], t\in[0,+\infty),\varepsilon\in(0,\varepsilon_{2})}\left(\frac{\phi \gamma_{2}}{\lambda_{1}}+\frac{\delta_{1}}{\phi\lambda_{2}}\chi_{\varepsilon}^{2}+\sqrt{\frac{g}{H_{1}}}\partial_{t} H_{1}\right)\\
&>\left(\frac{\phi \gamma_{2}}{\lambda_{1}}+\frac{\delta_{1}}{\phi\lambda_{2}}X^{2}+\sqrt{\frac{g}{H_{1}}}\delta+\varepsilon_{2}\right)\left(C_{6}\frac{\phi_{1}^{2}}{X^{2}}+C_{7}\phi_{2}^{2}\right)\left(\frac{X}{\phi_{1}\phi_{2}}\right)^{2}\delta\\
&+2\sqrt{\frac{g}{H_{1}}}\left(\frac{\phi\gamma_{2}}{\lambda_{1}}+\frac{\phi^{-1}\delta_{1}}{\lambda_{2}}X^{2}\right)\delta+\left(\frac{X}{\phi_{1}\phi_{2}}\right)^{2}C_{7}C_{6}\delta^{2},
 \end{split}
 \label{conddelta2}
 \end{equation} 
for any $x\in[0,L]$ and any $X\in[1/C_{5},C_{5}]$ (nota that having it for $X=C_{5}$ is enough). This is possible as $\varepsilon_{2}>0$ and, when $\delta_{3}=0$, 
\eqref{conddelta2} is verified and the inequality is strict. 
Then, from \eqref{phi}, \eqref{boundphi}, \eqref{cond00}, \eqref{boundf}--\eqref{conddelta2},
\begin{equation}
\begin{split}
(-(\lambda_{1}f_{1})_{x}+2\gamma_{1}f_{1}-\partial_{t}f_{1})((\lambda_{2}f_{2})_{x}+2\delta_{2}f_{2}-\partial_{t}f_{2})&>
\left(\frac{\phi_{1}\phi_{2}}{\chi_{\varepsilon_{2}}}\right)^{2}\left(\frac{\phi \gamma_{2}}{\lambda_{1}}+\frac{\delta_{1}}{\phi\lambda_{2}}\chi_{\varepsilon_{2}}^{2}\right)^{2}\\
&=\left(\frac{\gamma_{2}}{\lambda_{1}}f_{1}+\frac{\delta_{1}}{\lambda_{2}}f_{2}\right)^{2}\textcolor{black}{,}\\
\end{split}
\end{equation} 
which is exactly the second inequality of \eqref{c3}. Besides, 
from \eqref{positivity} and \eqref{conddelta02},
\begin{equation}
\begin{split}
(-(\lambda_{1}f_{1})_{x}+2\gamma_{1}f_{1}-\partial_{t}f_{1})=&\phi_{1}^{2}\frac{\partial_{x}\chi_{\varepsilon_{2}}}{\chi^{2}_{\varepsilon_{2}}}
-\partial_{t}f_{1}
\\
=&\frac{\phi_{1}^{2}}{\chi_{\varepsilon_{2}}^{2}}\left(
\left(\frac{\phi\gamma_{2}}{\lambda_{1}}+\frac{\delta_{1}}{\phi\lambda_{2}}f_{\textcolor{black}{\varepsilon_{2}}}^{2}+\sqrt{\frac{g}{H_{1}}}\partial_{t} H_{1}\right)+\textcolor{black}{\varepsilon_{2}}
-\frac{\partial_{t}f_{1}\chi_{\varepsilon_{2}}^{2}}{\phi_{1}^{2}}
\right)\\&>0.
\end{split}
\end{equation}

We can now check that \eqref{c1} and \eqref{c2} are also verified thanks to the choice of $\varepsilon_{2}$ and $\eta_{1}$. Indeed, using \eqref{eqf} and \eqref{k1k3}, one has
\begin{equation}
\frac{\lambda_{2}(0)f_{2}(t,0)}{\lambda_{1}(0)f_{1}(t,0)}=\chi_{\varepsilon_{2}}^{2}(t,0)=\left(\frac{\lambda_{2}(0)}{\lambda_{1}(0)}+\varepsilon_{2}\right)^{2}>\left(\frac{\lambda_{2}(0)}{\lambda_{1}(0)}\right)^{2}=k_{2}^{2}.
\end{equation} 
This explains our choice of initial condition for $\chi_{\varepsilon_{2}}$. 
Now, from \eqref{fL}, one has
\begin{equation}
\frac{\lambda_{1}(t,L)f_{1}(t,L)}{\lambda_{2}(t,L)f_{2}(t,L)}=\frac{\phi^{2}(t,L)}{\chi^{2}_{\varepsilon_{2}}(L)}>\frac{1}{\left(\frac{\lambda_{2}(t,L)}{\lambda_{1}(t,L)}+\eta_{1}\right)^{2}},
\end{equation} 
and from the definition of $\eta_{1}$ given by \eqref{defeta1},
\begin{equation}
\eta_{1}+\frac{\lambda_{2}(L)}{\lambda_{1}(L)}=\min\left(\frac{1}{|k_{1}|},1\right)\textcolor{black}{.}
\end{equation} 
Therefore, 
\begin{equation}
\frac{\lambda_{1}(t,L)f_{1}(t,L)}{\lambda_{2}(t,L)f_{2}(t,L)}>\max(k_{1}^{2},1),
\label{cond3}
\end{equation} 
and in particular
the condition \eqref{c2a} is verified. Let us now look at condition \eqref{c2b}. So far we have not selected the positive constant $q$. 
We want to show that there exists $q>0$ such that the condition \eqref{c2b} is satisfied. Observe that the left-hand side of \eqref{c2b} can be seen as a polynomial in $q$, 
and the condition \eqref{c2b} can be rewritten as
\begin{equation}
\begin{split}
P(q):=
&-\frac{q^{2}}{4}\frac{H_{1}}{g}\left(k_{1}-1\right)^{2}+q\sqrt{\frac{H_{1}}{g}}k_{3}\left(\lambda_{1}f_{1}(L)-\lambda_{2}f_{2}(L)(k_{1}^{2}-k_{1}(k_{1}-1))\right)
-\left(\lambda_{1}f_{1}(L)\right)\left(\lambda_{2}f_{2}(L)\right)k_{3}^{2}\\
=&-\frac{q^{2}}{4}\frac{H_{1}}{g}\left(k_{1}-1\right)^{2}+q\sqrt{\frac{H_{1}}{g}}k_{3}\left(\lambda_{1}f_{1}(L)-\lambda_{2}f_{2}(L)k_{1})\right)
-\left(\lambda_{1}f_{1}(L)\right)\left(\lambda_{2}f_{2}(L)\right)k_{3}^{2}>0\textcolor{black}{.}
\end{split}
\end{equation} 
From \eqref{cond3} $\lambda_{1}f_{1}(t,L)>\lambda_{2}f_{2}(t,L)k_{1}$ and from \eqref{condk1k3} $k_{3}>0$. Thus the real roots of $P$ are positive if they exist. This implies that there exists a positive constant $q$ such that \eqref{c2b} is satisfied if the discriminant of $P$ 
is positive. Denoting its discriminant by $\Delta$,
\begin{equation}
\Delta=\frac{H_{1}}{g}k_{3}^{2}\lambda_{2}^{2}f_{2}(t,L)^{2}\left[\left(\frac{\lambda_{1}f_{1}(L)}{\lambda_{2}f_{2}(L)}-k_{1}\right)^{2}-\left(\frac{\lambda_{1}f_{1}(L)}{\lambda_{2}f_{2}(L)}\right)\left(k_{1}-1\right)^{2}\right].
\end{equation} 
Let us introduce $h:X\rightarrow (X-k_{1})^{2}-X(k_{1}-1)^{2}$. The function $h$ is a second order polynomial with a positive dominant coefficient and observe that its roots are 
$k_{1}^{2}$ and $1$. Thus 
$h$ is increasing strictly on $[\max(k_{1}^{2},1),+\infty)$.
Hence, using \eqref{cond3},
\begin{equation}
\begin{split}
\Delta&=\frac{H_{1}}{g}k_{3}^{2}\lambda_{2}^{2}f_{2}(t,L)^{2}h(\frac{\lambda_{1}f_{1}(L)}{\lambda_{2}f_{2}(L)})\\
&>\frac{H_{1}}{g}k_{3}^{2}\lambda_{2}^{2}f_{2}(t,L)^{2}h(\max(k_{1}^{2},1))
=0.
\end{split}
\end{equation} 
This proves that there exists $q>0$ such that \eqref{c2b} is satisfied, and we select such $q$. 
All it remains to do now is to show that  the function $(\mathbf{U},z)\rightarrow V(t,\mathbf{U},z)$, which is now entirely selected, satisfies \eqref{cV1}.\\ 

From \eqref{fluvial} and \eqref{Hinfty} we know that for any $(t,x)\in[0,\infty)\times[0,L]$,
\begin{equation}
\sqrt{g H_{\max}}>\lambda_{2}>\alpha,\text{  }2\sqrt{gH_{\max}}>\lambda_{1}>\alpha.
\label{bornel1l2}
\end{equation} 
Besides, from the definition of $\phi_{1}$ and $\phi_{2}$ given by \eqref{phi}, \eqref{gamma} and the bound \eqref{fluvial}, \eqref{Hinfty}, there exists a constant $C_{8}$ that only depends on 
$\delta$, $\alpha$ and $H_{\max}$ such that
\begin{equation}
\frac{1}{C_{8}}\leq\lVert\phi_{1}\rVert_{\infty}\leq C_{8},\text{ }\frac{1}{C_{8}}\leq\lVert\phi_{2}\rVert_{\infty}\leq C_{8}.
\label{bornephi1phi2} 
\end{equation} 
Thus, using that $\chi_{0}=\lambda_{2}\phi/\lambda_{1}$, \eqref{deff1f2}, \eqref{fL}, \eqref{bornephi1phi2}, and \eqref{bornel1l2}, there exists $c_{1}>0$ constant independent of $\mathbf{U}$ and $z$ such that, for any $(\mathbf{U},z)\in H^{2}(0,L)\times \mathbb{R}$,
\begin{equation}
c_{1}\left(\lVert \mathbf{U}\rVert_{H^{2}(0,L)}+|z|\right) \leq V(t,(\mathbf{U},z))\leq \frac{1}{c_{1}}\left(\lVert \mathbf{U}\rVert_{H^{2}(0,L)}+|Z|\right)\forall\text{   }t\in[0,+\infty)\textcolor{black}{,}
\end{equation} 
which is exactly \eqref{cV1}. This concludes the proof of Theorem \ref{th1}.
\end{proof}
\section{Conclusion}
In this paper, we gave simple conditions on the design of a single PI controller to ensure 
the exponential stability of the nonlinear Saint-Venant equations with arbitrary friction and slope in the $H^{2}$ norm. 
These conditions apply when the inflow is an unknown constant, in that case the system has steady-states and any of them are stable. But they also apply when 
the inflow is time-dependent and slowly variable. In that case, no steady-states exists and one has to stabilize other target states. 
When the values of the target state are known at the end of the river, we have exponential stability of the target state. 
Otherwise, we have the Input-to-State stability with respect to the variation of the inflow disturbance. 
These sufficient conditions are found using a local quadratic entropy 
and, to the best of our knowledge, are less restrictive than any of the conditions that existed so far, even in the linear case.
In \cite{BastinCoronTamasoiu2015} it was shown that, in absence of friction and slope, these conditions were optimal for the linear case. However, so far there is no answer when there is some slope or friction and
whether these conditions are optimal or not would be a very interesting issue for a further study. \textcolor{black}{Its possible application to a network of channels would also be a matter of interest.
Finally, many stabilizing devices for finite dimensional systems also use a PID control with an additional derivative term. 
It has been shown in \cite{CoronTamasoiu2015} that this control cannot ensure exponential stability for a homogeneous hyperbolic equation. 
It would be an interesting question to know whether a filtering on the derivative term could enable to recover the stability for infinite dimensional system and whether this would enable a faster stabilization than the PI control.}
\section*{Acknowledgment}
The author would like to thank Jean-Michel Coron for his constant support and his advices. The author would like to thank Sebastien Boyaval for many fruitful discussions.
The author wishes also to thank Eric Demay, Peipei Shang, Shengquan Xiang and Christophe Zhang for fruitful discussions.
Finally the author would like to thank the ANR project Finite4SoS (ANR 15-CE23-0007) and the french Corps des IPEF.
\appendix

\section{Boundary conditions \eqref{bound} and \eqref{k1k3}}
\label{boundary}
In this appendix we justify the boundary conditions \eqref{bound} with \eqref{k1k3} after the change of variables. From the boundary conditions \eqref{bound1} in the physical coordinate $(h,v)$, together with the definition of $u_{1}$ and $u_{2}$
given in \eqref{change2}, one has at $x=L$
\begin{equation}
\begin{split}
u_{1}(t,L)=\mathcal{B}_{2}(h(t,L),Z(t),t)+\sqrt{\frac{g}{H_{1}}}h(t,L)=:\mathcal{F}_{1}(h(t,L),Z(t),x,t),\\
u_{2}(t,L)=\mathcal{B}_{2}(h(t,L),Z(t),t)-\sqrt{\frac{g}{H_{1}}}h(t,L)=:\mathcal{F}_{2}(h(t,L),Z(t),x,t).\\
\end{split}
\label{Bcond1}
\end{equation} 
From its definition, $\mathcal{F}_{1}$ is $C^{1}$ and, from \eqref{k1k30}, and \eqref{estimate}, there exists $\nu_{1}\in(0,\nu_{0})$ such that, for any $t\in[0,\infty)$, $\partial_{1}\mathcal{F}_{0}(0,Z(t),t)\neq 0$. Thus $\mathcal{F}_{1}$ is locally invertible with respect to its first variable, thus 
there exists $\nu_{2}\in(0,\textcolor{black}{\nu_{1}})$ such that
$h(t,L)=\mathcal{F}_{1}^{-1}(u_{1}(t,L),Z(t),t)$, where $\mathcal{F}_{1}^{-1}$ denotes the inverse with respect to the first variable. Besides, as $\mathcal{F}_{1}$ is of class $C^{2}$ with respect to the two first variables, $\mathcal{F}_{1}^{-1}$ is also of class $C^{2}$. Then, using \eqref{Bcond1}
\begin{equation}
u_{2}(t,L)=\mathcal{F}_{2}(\mathcal{F}_{1}^{-1}(u_{1}(t,L),Z(t),t),Z(t),t)=:\mathcal{D}_{2}(u_{1}(t,L),Z(t),t).
\end{equation} 
and, using \eqref{k1k30},
\begin{equation}
\begin{split}
 &\partial_{1}\mathcal{D}_{2}(0,0,t)=\partial_{1}\mathcal{F}_{2}(0,0,t)\partial_{1}(\mathcal{F}_{1}^{-1})(0,0,t)\\
 &=\frac{\partial_{1}\mathcal{F}_{2}(0,0,t)}{\partial_{1}\mathcal{F}_{1}(0,0,t)}=\frac{\partial_{1}\mathcal{B}_{2}(0,0,t)-\sqrt{\frac{g}{H_{1}}}}{\partial_{1}\mathcal{B}_{2}(0,0,t)+\sqrt{\frac{g}{H_{1}}}}\\
 &=-\frac{\lambda_{1}(L)-v_{G}(1+k_{p})}{\lambda_{2}(L)+v_{G}(1+k_{p})}.
 \end{split}
\end{equation} 
Now, as $\partial_{2}\mathcal{F}_{1}^{-1}(0,0,t)=-\partial_{2}\mathcal{F}_{1}(0,0,t)/\partial_{1}\mathcal{F}_{1}(0,0,t)$, using \eqref{k1k30},
\begin{equation}
\begin{split}
 &\partial_{2}\mathcal{D}_{2}(0,0,t)=\partial_{1}\mathcal{F}_{2}(0,0,t)\partial_{2}(\mathcal{F}_{1}^{-1})(0,0,t)+\partial_{2}\mathcal{F}_{2}(0,0,t)\\
 &=-\partial_{1}\mathcal{F}_{2}(0,0,t)\frac{\partial_{2}\mathcal{F}_{1}(0,0,t)}{\partial_{1}\mathcal{F}_{1}(0,0,t)}+\partial_{2}\mathcal{F}_{2}(0,0,t)\\
 &=\partial_{2}\mathcal{B}_{2}(0,0,t)\left(1-\frac{\partial_{1}\mathcal{B}_{2}(0,0,t)-\sqrt{\frac{g}{H_{1}}}}{\partial_{1}\mathcal{B}_{2}(0,0,t)+\sqrt{\frac{g}{H_{1}}}}\right)\\
 =&-\frac{v_{G}k_{I}}{H_{1}(t,L)}\left(\frac{2\sqrt{gH_{1}(t,L)}}{v_{G}(1+k_{p})+\lambda_{2}(t,L)}\right)\textcolor{black}{.}
\end{split}
\end{equation} 
The same can be done in $x=0$ in a slightly easier way, as $\mathcal{B}_{1}$ does not depends on $Z$. This gives \eqref{bound} and \eqref{k1k3}.

\section{Proof of Proposition \ref{ISS}}
\label{AppendixISS}
This appendix uses many computations that are very similar to the computations
in Section \ref{s2}, but in a simpler way. Thus, in order to avoid writing two times the same thing and to keep the proof relatively short, some steps might be quicker in this appendix.
Let $T_{1}>0$ and to be chosen later on. As $(H_{0}(0),V_{0}(0))$ satisfies \eqref{fluvial0}, there exists $\nu_{a}>0$ such that for $\nu\in(0,\nu_{a})$, $F((H_{1}^{0},V_{1}^{0})^{T})$ 
has two distinct nonzero eigenvalues. Recall that $F$ is given by \eqref{defF} and that that $\nu$ is the bound on $\lVert H_{1}^{0}-H_{0}(0),V_{1}^{0}-V_{0}(0)\rVert_{H^{2}(0,L)}$.
Besides, from \eqref{targetsteady2}, the function $(H_{0}(t,\cdot),V_{0}(t,\cdot))$ is the solution of a system of ODE with an initial condition depending on a parameter $t$. Thus, as $\partial_{t}Q_{0}\in C^{2}([0,+\infty))$ and the slope $C$ satisfies $C\in C^{2}([0,L])$,
using \eqref{targetsteady0} and \cite{Hartman}[Chap. 5, Theorem 3.1], $(H_{0},V_{0})\in C^{3}([0,T_{1}]; C^{2}([0,L]))$ and there exists a constant $C$ depending only on $H_{\max}$, $\alpha$ \textcolor{black}{and 
an upper bound of $\delta$,} such that, 
\textcolor{black}{\begin{equation}
\label{boundH0V0}
\lVert\partial_{t}^{i}H_{0},\partial_{t}^{i}V_{0}\rVert_{C^{2}([0,L])}\leq C\sum\limits_{n=1}^{i}\left|\partial_{t}^{n}Q_{0}\right|,\text{  }\forall\text{  }i\in[1,3],\text{ }\forall \text{ }t\in[0,T_{1}],
\end{equation} and in particular} 
\begin{equation}
\lVert\partial_{t} H_{0},\partial_{t} V_{0}\rVert_{C^{2}([0,T_{1}];C^{2}([0,L]))}\leq C \lVert\partial_{t}Q_{0}\rVert_{C^{\textcolor{black}{2}}([0,+\infty))}.
\label{H0V0bound}
\end{equation}
Thus \cite{Wang}[Theorem 2.1] can still be used 
\textcolor{black}{on $(H_{1}-H_{0})$}
and
there exist $\delta_{0}(T_{1})>0$ and $\nu_{0}(T_{1})\in(0,\nu_{a})$ such that, if $\nu\in(0,\nu_{0}(T_{1}))$ and $\delta\in(0,\delta_{0}(T_{1}))$, there exists a unique solution $(H_{1},V_{1})\in C^{0}([0,T_{1}];H^{2}(0,L))^{\textcolor{black}{2}}$ to the system \eqref{target}--\eqref{initialtarget}. 
Besides $(H_{1},V_{1})$ satifsfies an estimate as \eqref{estimate} but with $(H_{1},V_{1})$ instead of $(H,V)$ and $(H_{0},V_{0})$ instead of $(H_{1},V_{1})$. We denote by $C(T_{1})$ the associated constant.
Let us define $h_{1}:=H_{1}-H_{0}$ and $v_{1}:=V_{1}-V_{0}$.
We transform $(h_{1},v_{1})^{T}$ into $\mathbf{w}=(w_{1},w_{2})^{T}$ using the change of variables defined by 
\eqref{change1}--\eqref{change2} with $H_{0}$ and $V_{0}$ instead of $H_{1}$ and $V_{1}$. Thus we obtain
\begin{equation}
\begin{split}
\partial_{t}\mathbf{w}+A_{0}(\mathbf{w},x)\partial_{x}\mathbf{w}&+B_{0}(\mathbf{w},x)+S_{0}\begin{pmatrix}\partial_{t}H_{0}\\ \partial_{t}V_{0}\end{pmatrix}=0,\\
w_{1}(t,0)=&\mathcal{H}_{1}(w_{2}(t,0),Q_{0}(t)-Q_{0}(0)),\\
w_{2}(t,L)=&\mathcal{H}_{2}(w_{2}(t,L)),
\end{split}
\label{sysw}
\end{equation} 
where $A_{0}$, $B_{0}$ and $S_{0}$ have the same expression as $A$, $B$ and $S$ (given by\eqref{defA}, \eqref{defB}, \eqref{defS}) but with $(H_{0},V_{0})$ instead of $(H_{1}, V_{1})$. Similarly we define
\begin{equation}
\lambda_{1}^{0}=V_{0}+\sqrt{gH_{0}},\text{  }\lambda_{2}^{0}=\sqrt{gH_{0}}-V_{0}\textcolor{black}{,}
\end{equation} 
and $\phi^{0}$, defined as $\phi$ but with $(H_{0},V_{0})$ instead of $(H_{1}, V_{1})$.
Similarly as in Appendix \ref{boundary}, 
\begin{equation}
\mathcal{H}_{2}'(0)=-\lambda_{1}^{0}(L)/\lambda_{2}^{0}(L),\text{  }\mathcal{H}_{1}'(0)=-\lambda_{2}^{0}(0)/\lambda_{1}^{0}(0)\textcolor{black}{,}
\end{equation} 
which is of the form \eqref{bound} with $v_{G}=0$ and $Z=0$.
Before going any further, note that we can perform the same computations as in Section 2 with no problem,
as the proof in Section \ref{s2} only used Proposition \ref{propISS} to get that $(H_{1},V_{1})$ exists for any time and that \eqref{fluvial} and Lemma \ref{lem3} hold, but we will see now that such claims are true for $H_{0}$ and $V_{0}$. 
The existence of $(H_{0},V_{0})$ was already shown in section \ref{s1} and \eqref{fluvial0} is exactly \eqref{fluvial} with $(H_{0},V_{0})$ instead of $(H_{1},V_{1})$. Finally, \eqref{H0V0bound} is exactly the equivalent of Lemma \ref{lem3} for $(H_{0},V_{0})$.
We define now the Lyapunov fonction candidate $V:=V_{a}(\mathbf{w}(t,x),t)+V_{b}(\mathbf{w}(t,x),t)+V_{c}(\mathbf{w}(t,x),t)+V_{d}(\mathbf{w}(t,x),t)$ where $V_{a}$, $V_{b}$ and $V_{c}$ are defined in \eqref{Va}, \eqref{V2V3},
with $f_{1}$ and $f_{2}$ chosen as $f_{1}:=(\phi_{1}^{0})^{2}/(\lambda_{1}^{0}\eta)$ and $f_{2}:=(\phi_{2}^{0})^{2}\eta/(\lambda_{2}^{0})$, 
where $\eta$ is a function such that there exists a constant $\varepsilon>0$ independent of $\mathbf{w}$ such that
\begin{equation}
\begin{split}
 &\eta'=\left|\frac{\gamma_{2}^{0}}{\lambda_{1}^{0}}+\frac{\delta_{1}^{0}}{\lambda_{2}^{0}}\eta^{2}\right|+\varepsilon, \forall\text{  }x\in[0,L],\\
 &\eta(0)=\frac{\lambda_{2}^{0}(0)}{\lambda_{1}^{0}(0)}\phi^{0}(0)+\varepsilon.
 \end{split}
\end{equation} 
Note that $\eta$ exists as, for any $t\in[0,+\infty)$, $(\phi(t,\cdot)^{0}\lambda_{2}^{0}(t\cdot)/\lambda_{1}^{0}(t\cdot))$ is a solution of
\begin{equation}
\partial_{x} \chi=\left|\frac{\gamma_{2}^{0}}{\lambda_{1}^{0}}+\frac{\delta_{1}^{0}}{\lambda_{2}^{0}}\chi^{2}\right|, \forall\text{  }x\in[0,L],
\end{equation} 
this can be proved as in Lemma \ref{lem2}, and this case was actually shown in \cite{HS}.
Note that from \eqref{targetsteady0}, \eqref{targetsteady2} and \eqref{fluvial0}, $(H_{0})_{x}$ and $(V_{0})_{x}$ can be bounded by above and by below by constants that only depend on $H_{\max}$, $\alpha$ \textcolor{black}{and an upper bound of $Q_{0}$ (which can also be expressed only with $H_{\max}$, $\alpha$ from \eqref{fluvial0}).
}
Therefore, looking at their definition, the function $f_{1}$ and $f_{2}$ can also be bounded by above and below by constants that only depend on $H_{\max}$, $\alpha$ and $\varepsilon$. Thus there exist $c_{1}>0$ and $c_{2}>0$ depending only on $H_{\max}$ and $\alpha$, $\varepsilon$ and $\mu$ such that
\begin{equation}
c_{1}\lVert h_{1}(t,\cdot),v_{1}(t,\cdot)\rVert_{H^{2}(0,L)}^{2}\leq V(t)\leq c_{2}\lVert h_{1}(t,\cdot),v_{1}(t,\cdot)\rVert_{H^{2}(0,L)}^{2},\forall\text{  }t\in[0,T_{1}].
\label{equiv2}
\end{equation} 
Consequently, by differentiating $V$ exactly as in \eqref{diffVa0}--\eqref{dVtot}, and from \eqref{sysw}, we obtain that there exists $\mu>0$, $\nu_{1}\in(0,\nu_{0}(T_{1}))$ and $\textcolor{black}{\delta_{3}}>0$ such that, for any $\lVert h_{1}(0,\cdot),v_{1}(0,\cdot)\rVert_{H^{2}(0,L)}\leq \nu_{1}$,
and $\lVert\partial_{t}Q_{0}\rVert_{C^{2}([0,\infty))}\leq \delta$, \textcolor{black}{where $\delta\in(0,\delta_{3})$,} 
\begin{equation}
\begin{split}
\dot V\leq& -\mu V
+ \int_{0}^{L}2f_{1}w_{1}(S_{0}\begin{pmatrix}\partial_{t}H_{0}\\ \partial_{t}V_{0}\end{pmatrix})_{1}+2f_{2}w_{2}(S_{0}\begin{pmatrix}\partial_{t}H_{0}\\ \partial_{t}V_{0}\end{pmatrix})_{2}dx,\\
&+ \int_{0}^{L}2f_{1}\partial_{t}w_{1}(S_{0}\begin{pmatrix}\partial_{tt}^{2}H_{0}\\ \partial_{tt}^{2}V_{0}\end{pmatrix})_{1}+2f_{2}\partial_{t}w_{2}(S_{0}\begin{pmatrix}\partial_{tt}^{2}H_{0}\\ \partial_{tt}^{2}V_{0}\end{pmatrix})_{2}dx,\\
&+ \int_{0}^{L}2f_{1}\partial_{tt}^{2}w_{1}(S_{0}\begin{pmatrix}\partial_{ttt}^{3}H_{0}\\ \partial_{ttt}^{3}V_{0}\end{pmatrix})_{1}+2f_{2}\partial_{tt}^{2}w_{2}(S_{0}\begin{pmatrix}\partial_{ttt}^{3}H_{0}\\ \partial_{ttt}^{3}V_{0}\end{pmatrix})_{2}dx.
\end{split}
\label{additional}
\end{equation} 
Thus, using Cauchy-Schwarz inequality, \eqref{equiv2}, \textcolor{black}{and \eqref{boundH0V0}} there exists $C_{1}>0$ 
depending only on $H_{\max}$, $\alpha$ and an upper bound of $\mu$ such that
\textcolor{black}{
\begin{equation}
\dot V(t)\leq -\mu V(t)
+C_{1}\left(|\partial_{t}Q_{0}(t)|+|\partial_{tt}^{2}Q_{0}(t)|+|\partial_{ttt}^{3}Q_{0}(t)|\right)V^{1/2}(t),\text{ }\forall\text{ }t\in[0,T_{1}].
\label{H1V1decroissance}
\end{equation} 
and in particular
\begin{equation}
\textcolor{black}{\dot V(t)\leq -\mu V(t)
+C_{1}\lVert\partial_{t}Q_{0}\rVert_{C^{2}([0,\textcolor{black}{t]})}V^{1/2}(t),\text{ }\forall\text{ }t\in[0,T_{1}].}
\label{H1V1decroissance1}
 \end{equation} 
}
\textcolor{black}{ Let us define $V_{eq}:=(C_{1}\delta/\mu)^{2}$. From \eqref{H1V1decroissance1}, if $V(t)>2V_{eq}$, then there exists a constant $k>0$ such that $\dot V(t)<-k V^{1/2}(t)$.}
We now choose $\delta$ such that $\textcolor{black}{\sqrt{2}C_{1}\delta/(\mu \sqrt{c_{1}}})<\nu_{1}$. Thus,
from \eqref{H1V1decroissance1} and as $c_{1}$, $c_{2}$, $C_{1}$ and $\mu$ do not depend on $T_{1}$, we can choose $T_{1}$ large enough such that
\begin{equation}
 V(T_{1})\leq 2V_{eq}\leq c_{1}\nu_{1}^{2}\textcolor{black}{,}
\end{equation} 
which implies that 
\begin{equation}
\lVert h_{1}(T_{1},\cdot),v_{1}(T_{1},\cdot)\rVert_{C^{2}(0,L)}\leq \nu_{1}
\end{equation} 
and therefore there exists a unique solution $(h_{1},v_{1})\in C^{0}([T_{1},2T_{1}],H^{2}(0,L))$, with initial condition $(h_{1}(T_{1},\cdot),v_{1}(T_{1},\cdot))$ (we use the same existence Theorem (\cite{Wang}[Theorem 2.1])) and, 
\textcolor{black}{noting that $V(T_{1})\leq 2V_{eq}$ implies $V(2T_{1})\leq 2V_{eq}$,
this analysis still hold. }
We can do similarly for any $[nT_{1},(n+1)T_{1}]$ with $n\in\mathbb{N}$, thus, as $(H_{0},V_{0})\in C^{0}([0,+\infty),H^{2}(0,L))$, there exists a unique solution $(H_{1},V_{1})\in C^{0}([0,+\infty),H^{2}(0,L))$
and \eqref{H1V1decroissance} holds for any $t\in[0,+\infty)$. 
\textcolor{black}{
Therefore, denoting $g(t)=V(t)e^{\mu t}$,
we deduce from \eqref{H1V1decroissance} that 
\begin{equation}
g'(t)\leq C_{1}\left(|\partial_{t}Q_{0}(t)|+|\partial_{tt}^{2}Q_{0}(t)|+|\partial_{ttt}^{3}Q_{0}(t)|\right)e^{\frac{\mu t}{2}}\sqrt{g(t)}.
\label{ISS0}
\end{equation} 
Thus 
\begin{equation}
V^{1/2}(t)\leq V^{1/2}(0)e^{-\frac{\mu t}{2}}+\frac{C_{1}}{2}\left(\int_{0}^{t}\left(|\partial_{t}Q_{0}(t)|+|\partial_{tt}^{2}Q_{0}(t)|+|\partial_{ttt}^{3}Q_{0}(t)|\right)e^{\frac{\mu s}{2}}ds\right)e^{-\frac{\mu t}{2}}.
\end{equation} 
}
This implies the ISS property
\textcolor{black}{\begin{equation}
\begin{split}
\lVert h_{1}(t,\cdot),v_{1}(t,\cdot)\rVert_{H^{2}((0,L);\mathbb{R}^{2})}\leq& \sqrt{\frac{c_{2}}{c_{1}}}\lVert h_{1}(0,\cdot),v_{1}(0,\cdot)\rVert_{H^{2}((0,L);\mathbb{R}^{2})}e^{-\frac{\mu t}{2}}\\
&+\frac{C_{1}}{2\sqrt{c_{1}}}\left(\int_{0}^{t}\left(|\partial_{t}Q_{0}(t)|+|\partial_{tt}^{2}Q_{0}(t)|+|\partial_{ttt}^{3}Q_{0}(t)|\right)e^{\frac{\mu s}{2}}ds\right)e^{-\frac{\mu t}{2}}.
\end{split}
\label{eqISS}
\end{equation}}
 This ends the proof of Proposition \ref{propISS}. 
 To extend this proof to the $H^{p}$ norm for $p>2$, note that using the same argument
 \eqref{H0V0bound} holds with the $C^{p}([0,T_{1}];C^{3}([0,L]))$ norm in the left-hand side and the $C^{p}$ norm in the right-hand side. 
We can can define $V_{3},...,V_{p}$ on $H^{p}(0,L)\times\mathbb{R}\times\mathbb{R}_{+}$ as in \eqref{V2V3} such that $V_{k}(\mathbf{w}(t,x),t)=V_{a}(\partial_{t}^{k}\mathbf{w}(t,x),t)$, for any $k\in[3,p]$.
Then \eqref{equiv2} holds with $V:=V_{a}+V_{b}+V_{c}+V_{3}+...V_{p}$ and the $H^{p}$ norm, and the rest can done done identically.
\section{Proof of Theorem \ref{th2}}
\label{ISS2}
Theorem \ref{th2} result from the proof of Theorem \ref{th1}. Note that the boundary conditions \eqref{boundary2} can be written under the form \eqref{boundary1} with $(H_{0},V_{0})$ instead of $(H_{1},V_{1})$
where the only difference is that $Z$ satisfies now
\begin{equation}
\dot Z=H_{c}-H(t,L)+\frac{f(t)}{v_{G}k_{I}},
\label{Z2}
\end{equation} 
where $f(t)=H_{c}\partial_{t}V_{0}(t,L)$. 
The rest of the proof can be conducted as in Appendix \ref{AppendixISS} for $(H_{1},V_{1})$, 
with $a$ $priori$ two differences: $(H,V)$ satisfies the boundary conditions of the form \eqref{boundary1} and not of the form given in \eqref{target},
and $\dot Z$ satisfies \eqref{Z2} instead of \eqref{Z}. However,
note that in Appendix \ref{AppendixISS} the only assumption used on the boundary conditions of the transformed system is that they are of the form \eqref{bound1}, 
which is still the case here.
Thus,
the only difference with Appendix \ref{AppendixISS} are
some additional terms when $\dot Z$ is used, which is in the boundary terms in the derivative of the Lyapunov function. There exists therefore $\delta_{4}>0$ and $\nu_{2}>0$ such that\textcolor{black}{, for any $\lVert h_{1}(0,\cdot),v_{1}(0,\cdot)\rVert_{H^{2}(0,L)}\leq \nu_{2}$,
and $\lVert\partial_{t}Q_{0}\rVert_{C^{2}([0,\infty))}\leq \delta$, where $\delta\in(0,\delta_{4})$,}
\begin{equation}
\textcolor{black}{\dot V(t)\leq -\frac{\gamma}{2}V(t)
\textcolor{black}{+C_{1}|\partial_{t}Q_{0}(t)+\partial_{tt}^{2}Q_{0}(t)+\partial_{ttt}^{3}Q_{0}(t)|V^{1/2}}+2qZf(t)+2q\dot Zf'(t)+2q\ddot Zf''(t),}
\end{equation} 
where $C_{1}$ is a constant only depending on $H_{\max}$, $\alpha$, $\nu_{2}$ and $\delta_{4}$. Using Lemma \ref{lem3}, 
there exists a constant $C>0$ depending only on $H_{\max}$, $\alpha$, $\nu_{2}$ and $\delta_{4}$ such that
\begin{equation}
\dot V\leq -\frac{\gamma}{2}V+CV^{1/2}\textcolor{black}{|\partial_{t}Q_{0}(t)+\partial_{tt}^{2}Q_{0}(t)+\partial_{ttt}^{3}Q_{0}(t)|}.
\end{equation} 
The same argument as in Appendix \ref{AppendixISS}, \eqref{ISS0}--\eqref{eqISS}, implies directly the ISS property \eqref{ISSHV}.

\section{Proof of Lemma \ref{lem2}}
In this appendix we prove Lemma \ref{lem2}. The proof is very similar to the proof given in \cite{HS} in the special case where $(H_{1},V_{1})$ is a steady state. However, it happens that the proof actually does not need the relation $(H_{1}V_{1})_{x}=0$ which is no longer true when $(H_{1},V_{1})$ is not a steady-state. Let $\chi=(\lambda_{2}\phi/\lambda_{1})$, we have from \eqref{phi}:
\begin{equation}
\begin{split}
\partial_{x}\chi=&\frac{\phi}{\lambda_{1}^{2}}\left(\lambda_{1}\partial_{x}\lambda_{2}-\lambda_{2}\partial_{x}\lambda_{1}+\lambda_{2}\gamma_{1}+\lambda_{1}\delta_{2}\right)\\
=&\frac{\phi}{\lambda_{1}^{2}}\left((V_{1}+\sqrt{gH_{1}})(-V_{1x}+\frac{\sqrt{gH_{1}}}{2H_{1}}H_{1x})-(-V_{1}+\sqrt{gH_{1}})(V_{1x}+\frac{\sqrt{gH_{1}}}{2H_{1}}H_{1x})\right.\\
&\left.+(\sqrt{gH_{1}}-V_{1})\left(\frac{3}{4}\sqrt{\frac{g}{H_{1}}}H_{1x}+\frac{3}{4}V_{1x}+\frac{kV}{H_{1}}-\frac{kV_{1}^{2}}{2H_{1}^{2}}\sqrt{\frac{H_{1}}{g}}\right)\right.\\
&\left.+(V_{1}+\sqrt{gH_{1}})\left(-\frac{3}{4}\sqrt{\frac{g}{H_{1}}}H_{1x}+\frac{3}{4}V_{1x}+\frac{2kV}{H_{1}}+\frac{kV_{1}^{2}}{2H_{1}^{2}}\sqrt{\frac{H_{1}}{g}}\right)\right)\\
=&\frac{\phi}{\lambda_{1}^{2}}\left(\sqrt{gH_{1}}\left(-2V_{1x}+\frac{3}{2}V_{1x}+\frac{2kV}{H_{1}}\right)-V_{1}\left(\frac{3}{2}\sqrt{\frac{g}{H_{1}}}H_{1x}-\frac{kV_{1}^{2}}{H_{1}^{2}}\sqrt{\frac{H_{1}}{g}}-\sqrt{\frac{g}{H_{1}}}H_{1x}\right)\right)\\
=&\frac{\phi}{\lambda_{1}^{2}}\left(\frac{2kV}{H_{1}}\sqrt{gH_{1}}+\frac{kV_{1}^{2}}{H_{1}^{2}}\sqrt{\frac{H_{1}}{g}}V_{1}+\frac{1}{2}\sqrt{\frac{g}{H_{1}}}\partial_{t}H_{1}\right)\textcolor{black}{.}
\end{split}
\label{Le1}
\end{equation}
And on the other hand:
\begin{equation}
\begin{split}
\left(\frac{\phi\gamma_{2}}{\lambda_{1}}+\frac{\delta_{1}}{\lambda_{2}\phi}\chi^{2}\right)&=\frac{\phi}{\lambda_{1}^{2}}\left(\lambda_{1}\gamma_{2}+\lambda_{2}\delta_{1}\right)\\
&=\frac{\phi}{\lambda_{1}^{2}}\left(\frac{2kV}{H_{1}}\sqrt{gH_{1}}+\frac{kV_{1}^{2}}{H_{1}^{2}}\sqrt{\frac{H_{1}}{g}}V_{1}+V_{1}\sqrt{\frac{g}{H_{1}}}\frac{H_{1x}}{2}+V_{1x}\frac{\sqrt{gH_{1}}}{2}\right)\\
&=\frac{\phi}{\lambda_{1}^{2}}\left(\frac{2kV}{H_{1}}\sqrt{gH_{1}}+\frac{kV_{1}^{2}}{H_{1}^{2}}\sqrt{\frac{H_{1}}{g}}V_{1}-\frac{1}{2}\sqrt{\frac{g}{H_{1}}}\partial_{t}H_{1}\right).
\end{split}
\label{Le2}
\end{equation}
Thus from \eqref{Le1} and \eqref{Le2}
\begin{equation}
 \partial_{x}\chi=\left(\frac{\phi\gamma_{2}}{\lambda_{1}}+\frac{\delta_{1}}{\lambda_{2}\phi}\chi^{2}+\sqrt{\frac{g}{H_{1}}}\partial_{t}H_{1}\right).
 \label{Leq}
\end{equation} 
And there exists $\delta_{0}$ such that, if $\lVert \partial_{t}H_{1}\rVert_{L^{\infty}((0,+\infty)\times(0,L)}\delta_{0}$,
\begin{equation}
\frac{\phi}{\lambda_{1}^{2}}\left(\frac{2kV_{1}}{H_{1}}\sqrt{gH_{1}}+\frac{kV_{1}^{2}}{H_{1}^{2}}\sqrt{\frac{H_{1}}{g}}V_{1}+\sqrt{\frac{g}{H_{1}}}\partial_{t}H_{1}\right)>0,\text{ }\forall\text{ }x\in[0,L],\text{  }t\in[0,+\infty)\textcolor{black}{,}
\end{equation} 
and, from \eqref{Le1} and \eqref{Leq},
\begin{equation}
\partial_{x}\chi=\left|\frac{\phi\gamma_{2}}{\lambda_{1}}+\frac{\delta_{1}}{\lambda_{2}\phi}\chi^{2}+\sqrt{\frac{g}{H_{1}}}\partial_{t}H_{1}\right|,
\end{equation} 
this ends the proof of Lemma \ref{lem2}.

\section{Proof of Lemma \ref{lem3}}\label{app:prooflem3}
In this Appendix we show that Lemma \ref{lem3} is a consequence of Proposition \ref{propISS} and Remark \ref{rmkHpISS}.
\begin{proof}
Indeed using Proposition \ref{propISS} and Remark \ref{rmkHpISS} with $p=3$, we have
\begin{equation}
\begin{split}
\|H_{1}(t,\cdot)-H_{0}(t,\cdot)\|_{H^{3}(0,L)}+\|V_{1}(t,\cdot)-V_{0}(t,\cdot)\|_{H^{3}(0,L)}\leq& \left(\|H_{1}^{0}-H^{*}\|_{H^{3}(0,L)}+\|V_{1}^{0}-V^{*}\|_{H^{3}(0,L)}\right)e^{-\frac{\mu t}{2}}\\
&+c_{2} \frac{2}{\mu}(1-e^{-\frac{\mu t}{2}}) \|\partial_{t}Q_{0}\|_{C^{3}([0,+\infty))}.
\end{split}
\end{equation}
Note that we chose $H_{1}^{0} = H^{*}$ and $V_{1}^{0} = V^{*}$ which means that
\begin{equation}\label{estimlemma1}
\|H_{1}(t,\cdot)-H_{0}(t,\cdot)\|_{H^{3}(0,L)}+\|V_{1}(t,\cdot)-V_{0}(t,\cdot)\|_{H^{3}(0,L)}\leq c_{2} \frac{2}{\mu}\|\partial_{t}Q_{0}\|_{C^{3}([0,+\infty))}.
\end{equation}
\textcolor{black}{Note that $H_{1}-H_{1}^{0}$ is the solution of a quasilinear hyperbolic system and is small in $H^{3}$ norm provided that $\partial_{t}Q_{0}$ is small in $C^{3}$ norm. Therefore, there exists a constant $C$ depending only on the parameters of the system and the bound $\nu$ such that
\begin{equation}
\label{estimlemma2}
\begin{split}
\|H_{1}(t,\cdot)-H_{0}(t,\cdot)\|_{L^{2}(0,L)}+\|V_{1}(t,\cdot)-V_{0}(t,\cdot)\|_{L^{2}(0,L)}&\\
+\|\partial_{t}H_{1}(t,\cdot)-\partial_{t}H_{0}(t,\cdot)\|_{L^{2}(0,L)}+\|\partial_{t}V_{1}(t,\cdot)-\partial_{t}V_{0}(t,\cdot)\|_{L^{2}(0,L)}&\\
+\|\partial_{tx}^{2}H_{1}(t,\cdot)-\partial_{tx}^{2}H_{0}(t,\cdot)\|_{L^{2}(0,L)}+\|\partial_{tx}^{2}V_{1}(t,\cdot)-\partial_{tx}^{2}V_{0}(t,\cdot)\|_{L^{2}(0,L)}&\\
+\|\partial_{tt}^{2}H_{1}(t,\cdot)-\partial_{tt}^{2}H_{0}(t,\cdot)\|_{L^{2}(0,L)}+\|\partial_{tt}^{2}V_{1}(t,\cdot)-\partial_{tt}^{2}V_{0}(t,\cdot)\|_{L^{2}(0,L)}&\\
+\|\partial_{ttt}^{3}H_{1}(t,\cdot)-\partial_{ttt}^{3}H_{0}(t,\cdot)\|_{L^{2}(0,L)}+\|\partial_{ttt}^{3}V_{1}(t,\cdot)-\partial_{ttt}^{3}V_{0}(t,\cdot)\|_{L^{2}(0,L)}&\\
+\|\partial_{ttx}^{3}H_{1}(t,\cdot)-\partial_{ttx}^{3}H_{0}(t,\cdot)\|_{L^{2}(0,L)}+\|\partial_{ttx}^{3}V_{1}(t,\cdot)-\partial_{ttx}^{3}V_{0}(t,\cdot)\|_{L^{2}(0,L)}
&\\
\leq C\|H_{1}(t,\cdot)-H_{0}(t,\cdot)\|_{H^{3}(0,L)}+\|V_{1}(t,\cdot)-V_{0}(t,\cdot)\|_{H^{3}(0,L)}&
\end{split}
\end{equation}
}
In what follows, the value of $C$ might change between lines but it still denotes a constant that only depends on the parameters of the system and the bound $\nu$. Besides, from Sobolev inequality, for $f\in H^{1}([0,L])$
\begin{equation}\label{estimlemma3}
\|f\|_{C^{0}([0,L])} \leq C\left(\|f\|_{L^{2}([0,L])}+\|\partial_{x}f\|_{L^{2}([0,L])}\right).
\end{equation}
Combining \eqref{estimlemma1}, \eqref{estimlemma2} and \eqref{estimlemma3},
\begin{equation}
\begin{split}
&\max(\|H_{1}(t,\cdot)-H_{0}(t,\cdot)\|_{C^{0}([0,L])},\|V_{1}(t,\cdot)-V_{0}(t,\cdot)\|_{C^{0}([0,L])})\\
&+\max(\|\partial_{t}H_{1}(t,\cdot)-\partial_{t}H_{0}(t,\cdot)\|_{L^{\infty}([0,L])},\|\partial_{t}V_{1}(t,\cdot)-\partial_{t}V_{0}(t,\cdot)\|_{C^{0}([0,L])})\\
&+
\max(\|\partial_{tt}^{2}H_{1}(t,\cdot)-\partial_{tt}^{2}H_{0}(t,\cdot)\|_{L^{\infty}([0,L])},\|\partial_{tt}V_{1}(t,\cdot)-\partial_{tt}V_{0}(t,\cdot)\|_{C^{0}([0,L])})
\\
&\leq Cc_{2} \frac{2}{\mu} \|\partial_{t}Q_{0}\|_{C^{3}([0,+\infty))},
\end{split}
\end{equation}
Therefore, using the inverted triangular inequality and the fact that 
$$\max(\|H_{1}(t,\cdot)-H_{0}(t,\cdot)\|_{C^{0}([0,L])},\|V_{1}(t,\cdot)-V_{0}(t,\cdot)\|_{C^{0}([0,L])})>0,$$ 
we have
\begin{equation}
\begin{split}
&\max(\|\partial_{t}H_{1}(t,\cdot)\|_{L^{\infty}([0,L])},\|\partial_{t}V_{1}(t,\cdot)\|_{C^{0}([0,L])})+
\max(\|\partial_{tt}^{2}H_{1}(t,\cdot)\|_{L^{\infty}([0,L])},\|\partial_{tt}V_{1}(t,\cdot)-\|_{C^{0}([0,L])})
\\
&\leq \max(\|\partial_{t}H_{0}(t,\cdot)\|_{L^{\infty}([0,L])},\|\partial_{t}V_{0}(t,\cdot)\|_{C^{0}([0,L])})+
\max(\|\partial_{tt}^{2}H_{0}(t,\cdot)\|_{L^{\infty}([0,L])},\|\partial_{tt}V_{0}(t,\cdot)-\|_{C^{0}([0,L])})\\
&+ Cc_{2} \frac{2}{\mu} \|\partial_{t}Q_{0}\|_{C^{3}([0,+\infty))}.
\end{split}
\end{equation}
Recall that $(H_{0}(t,\cdot),V_{0}(t,\cdot))$ satisfies \eqref{boundH0V0},
This means that
\begin{equation}
\begin{split}
&\max(\|\partial_{t}H_{1}(t,\cdot)\|_{L^{\infty}([0,L])},\|\partial_{t}V_{1}(t,\cdot)\|_{C^{0}([0,L])})+
\max(\|\partial_{tt}^{2}H_{1}(t,\cdot)\|_{L^{\infty}([0,L])},\|\partial_{tt}V_{1}(t,\cdot)-\|_{C^{0}([0,L])})
\\
&\leq C(1+c_{2} \frac{2}{\mu}) \|\partial_{t}Q_{0}\|_{C^{3}([0,+\infty))},
\end{split}
\end{equation}
As this is true for any $t\in [0,+\infty)$ we have
\begin{equation}
\|\partial_{t}H_{1},\partial_{t}V_{1}\|_{C^{1}([0,+\infty),C^{0}([0,L]))}\leq (1+Cc_{2} \frac{2}{\mu}) \|\partial_{t}Q_{0}\|_{C^{3}([0,+\infty))}.
\end{equation}
This ends the proof of Lemma \ref{lem3}.
\end{proof}
\bibliographystyle{plain}
\bibliography{SV_PI}

\begin{thebibliography}{10}

\bibitem{ABG}
Fabio {Ancona}, Alberto {Bressan}, and Giuseppe~Maria {Coclite}.
\newblock {Some results on the boundary control of systems of conservation
  laws.}
\newblock In {\em {H}yperbolic problems: Theory, numerics, applications.
  {P}roceedings of the ninth international conference on hyperbolic problems,
  {P}asadena, {CA}, {USA}, {M}arch 25--29, 2002}, pages 255--264. Berlin:
  Springer, 2003.

\bibitem{Astrom2}
Karl~Johan {\AA}str\"om and Richard~M. Murray.
\newblock {\em Feedback systems}.
\newblock Princeton University Press, Princeton, NJ, 2008.
\newblock An introduction for scientists and engineers.

\bibitem{SaintVenant}
Adhémar Barr\'{e}~de Saint-Venant.
\newblock Th\'{e}orie du mouvement non permanent des eaux, avec application aux
  crues des rivi\`{e}res et \`{a} l'introduction des mar\'{e}es dans leurs
  lits.
\newblock {\em Comptes Rendus des s\'{e}ances de l'Acad\'{e}mie des Sciences},
  73:237--240, 1871.

\bibitem{BastinCoron22}
Georges Bastin and Jean-Michel Coron.
\newblock On boundary feedback stabilization of non-uniform linear 2$\times$2
  hyperbolic systems over a bounded interval.
\newblock {\em Systems \& Control Letters}, 60(11):900--906, 2011.

\bibitem{BastinCoron2013}
Georges Bastin and Jean-Michel Coron.
\newblock {Exponential stability of networks of density-flow conservation laws
  under PI boundary control}.
\newblock {\em IFAC Proceedings Volumes}, 46(26):221--226, 2013.

\bibitem{BastinCoron1D}
Georges Bastin and Jean-Michel Coron.
\newblock {\em Stability and boundary stabilization of 1-d hyperbolic systems}.
\newblock Springer, 2016.

\bibitem{BC2017}
Georges Bastin and Jean-Michel Coron.
\newblock A quadratic {L}yapunov function for hyperbolic density-velocity
  systems with nonuniform steady states.
\newblock {\em Systems \& Control Letters}, 104:66--71, 2017.

\bibitem{BastinCoronPI}
Georges Bastin and Jean-Michel Coron.
\newblock {Exponential stability of PI control for Saint-Venant equations with
  a friction term}.
\newblock {\em Automatica}, 100:52, 2018.

\bibitem{stvenantjump}
Georges Bastin, Jean-Michel Coron, Amaury Hayat, and Peipei Shang.
\newblock Boundary feedback stabilization of hydraulic jumps.
\newblock {\em Preprint}, 2018.

\bibitem{BastinCoronTamasoiu2015}
Georges Bastin, Jean-Michel Coron, and Simona~Oana Tamasoiu.
\newblock Stability of linear density-flow hyperbolic systems under {PI}
  boundary control.
\newblock {\em Automatica J. IFAC}, 53:37--42, 2015.

\bibitem{BG}
Alberto Bressan and Giuseppe~Maria Coclite.
\newblock On the boundary control of systems of conservation laws.
\newblock {\em SIAM J. Control Optim.}, 41(2):607--622, 2002.

\bibitem{coron1999}
Jean-Michel Coron, Brigitte d'Andr\'ea Novel, and Georges Bastin.
\newblock A {L}yapunov approach to control irrigation canals modeled by
  {S}aint-{V}enant equations.
\newblock In {\em CD-Rom Proceedings, Paper F1008-5, ECC99, Karlsruhe,
  Germany}, pages 3178--3183, 1999.

\bibitem{PItransport}
Jean-Michel Coron and Amaury Hayat.
\newblock {PI controllers for 1-D nonlinear transport equation}.
\newblock {\em IEEE Transactions on Automatic Control}, 64(11):4570--4582,
  2019.

\bibitem{CoronFredholm}
Jean-Michel Coron, Long Hu, and Guillaume Olive.
\newblock Finite-time boundary stabilization of general linear hyperbolic
  balance laws via {F}redholm backstepping transformation.
\newblock {\em Automatica J. IFAC}, 84:95--100, 2017.

\bibitem{CoronTamasoiu2015}
Jean~Michel Coron and Simona~Oana Tamasoiu.
\newblock Feedback stabilization for a scalar conservation law with {PID}
  boundary control.
\newblock {\em Chin. Ann. Math. Ser. B}, 36(5):763--776, 2015.

\bibitem{CVBK}
Jean-Michel Coron, Rafael Vazquez, Miroslav Krstic, and Georges Bastin.
\newblock Local exponential {$H^2$} stabilization of a {$2\times 2$}
  quasilinear hyperbolic system using backstepping.
\newblock {\em SIAM J. Control Optim.}, 51(3):2005--2035, 2013.

\bibitem{DosSantosBastinCoron}
V.~Dos~Santos, G.~Bastin, J.-M. Coron, and B.~d'Andr\'{e}a Novel.
\newblock Boundary control with integral action for hyperbolic systems of
  conservation laws: stability and experiments.
\newblock {\em Automatica J. IFAC}, 44(5):1310--1318, 2008.

\bibitem{dos2008}
Val{\'e}rie Dos~Santos and Christophe Prieur.
\newblock Boundary control of open channels with numerical and experimental
  validations.
\newblock {\em IEEE transactions on Control systems technology},
  16(6):1252--1264, 2008.

\bibitem{Hartman}
Philip Hartman.
\newblock {\em Ordinary differential equations}.
\newblock John Wiley \& Sons, Inc., New York-London-Sydney, 1964.

\bibitem{C1}
Amaury Hayat.
\newblock {Boundary Stability of 1-D Nonlinear Inhomogeneous Hyperbolic Systems
  for the $C^{1}$ Norm}.
\newblock {\em SIAM Journal on Control and Optimization}, 57(6):3603--3638,
  2019.

\bibitem{C1_22}
Amaury Hayat.
\newblock {On boundary stability of inhomogeneous 2$\times$ 2 1-D hyperbolic
  systems for the $C^{1}$ norm}.
\newblock {\em ESAIM: Control, Optimisation and Calculus of Variations}, 25:82,
  2019.

\bibitem{HS}
Amaury Hayat and Peipei Shang.
\newblock A quadratic {L}yapunov function for {S}aint-{V}enant equations with
  arbitrary friction and space-varying slope.
\newblock {\em Automatica J. IFAC}, 100:52--60, 2019.

\bibitem{IassonKrstic}
Iasson Karafyllis and Miroslav Krstic.
\newblock {\em Input-to-State Stability for PDEs}.
\newblock Springer, 2018.

\bibitem{krsticbook}
Miroslav Krstic and Andrey Smyshlyaev.
\newblock {\em Boundary {C}ontrol of {PDE}s: A {C}ourse on {B}ackstepping
  {D}esigns}, volume~16 of {\em Advances in Design and Control}.
\newblock Society for Industrial and Applied Mathematics (SIAM), Philadelphia,
  PA, 2008.

\bibitem{Lichtner}
Mark Lichtner.
\newblock Spectral mapping theorem for linear hyperbolic systems.
\newblock {\em Proc. Amer. Math. Soc.}, 136(6):2091--2101, 2008.

\bibitem{Neves}
Aloisio~Freiria Neves, Hermano de~Souza Ribeiro, and Orlando Lopes.
\newblock On the spectrum of evolution operators generated by hyperbolic
  systems.
\newblock {\em J. Funct. Anal.}, 67(3):320--344, 1986.

\bibitem{PrieurISS}
Christophe Prieur and Fr\'{e}d\'{e}ric Mazenc.
\newblock I{SS}-{L}yapunov functions for time-varying hyperbolic systems of
  balance laws.
\newblock {\em Math. Control Signals Systems}, 24(1-2):111--134, 2012.

\bibitem{Sontag}
Eduardo~D. Sontag.
\newblock Input to state stability: basic concepts and results.
\newblock In {\em Nonlinear and optimal control theory}, volume 1932 of {\em
  Lecture Notes in Math.}, pages 163--220. Springer, Berlin, 2008.

\bibitem{TerrandAndrieuDos}
Alexandre Terrand-Jeanne, Vincent Andrieu, Valérie Dos Santos~Martins, and
  Cheng-Zhong Xu.
\newblock Adding integral action for open-loop exponentially stable semigroups
  and application to boundary control of pde systems.
\newblock {\em Preprint}, 2018.

\bibitem{Trinh2015}
Ngoc-Tu Trinh, Vincent Andrieu, and Cheng-Zhong Xu.
\newblock Pi regulation control of a fluid flow model governed by hyperbolic
  partial differential equations.
\newblock In {\em Proceeding of the International Conference on System
  Engineering, Coventry, England}, 2015.

\bibitem{Wang}
Zhiqiang Wang.
\newblock Exact controllability for nonautonomous first order quasilinear
  hyperbolic systems.
\newblock {\em Chinese Ann. Math. Ser. B}, 27(6):643--656, 2006.

\bibitem{XuSallet}
Cheng-Zhong Xu and Gauthier Sallet.
\newblock Proportional and integral regulation of irrigation canal systems
  governed by the st venant equation.
\newblock {\em IFAC Proceedings Volumes}, 32(2):2274--2279, 1999.

\bibitem{XuSallet2014}
Cheng-Zhong Xu and Gauthier Sallet.
\newblock Multivariable boundary {PI} control and regulation of a fluid flow
  system.
\newblock {\em Math. Control Relat. Fields}, 4(4):501--520, 2014.

\bibitem{Zhang}
Christophe Zhang.
\newblock {Internal rapid stabilization of a 1-D linear transport equation with
  a scalar feedback}.
\newblock October 2018.
\newblock working paper or preprint.

\bibitem{Zhang2}
Christophe Zhang.
\newblock {Finite-time internal stabilization of a linear 1-D transport
  equation}.
\newblock {\em Systems \& Control Letters}, 133:104529, 2019.

\end{thebibliography}
\end{document}